\documentclass[10pt]{amsart}

\usepackage[T1]{fontenc}
\usepackage{amsmath,amsthm,amsfonts,mathtools}
\usepackage{fullpage}
\usepackage{xcolor}
\usepackage{ulem}
\usepackage{amsaddr}

\theoremstyle{plain}
\newtheorem{theorem}{Theorem}
\newtheorem{prop}{Proposition}
\newtheorem{lemma}{Lemma}
\newtheorem{corollary}{Corollary}

\theoremstyle{definition}
\newtheorem{definition}{Definition}
\newtheorem{remark}{Remark}

\newtheorem{assumption}{Assumption}

\def \T {\mathbb{T}}
\def \Z {\mathbb{Z}}
\def \R {\mathbb{R}}
\def \P {\mathbb{P}}
\def \N {\mathbb{N}}
\def \E {\mathbb{E}}
\def \dd {\mathrm{d}}
\def \e {\mathrm{e}}
\def \sgn {\mathrm{sign}}

\def \DD {\mathrm{D}}
\newcommand{\ind}[1]{\mathbf{1}_{#1}}

\renewcommand{\tilde}{\widetilde}
\renewcommand{\bar}{\overline}

\title{Viscous scalar conservation law with stochastic forcing: strong solution and invariant measure}

\thanks{This work is partially supported by the French National Research Agency (ANR) under the programs ANR-17-CE40-0030 - EFI - Entropy, flows, inequalities and QuAMProcs.}

\author{Sofiane Martel}
\address{INRIA Rennes - Bretagne Atlantique, 35042 Rennes, France.}
\email{sofiane.martel@inria.fr}

\author{Julien Reygner}
\address{Universit\'e Paris-Est, CERMICS (ENPC), 77455 Marne-la-Vall\'ee, France.}
\email{julien.reygner@enpc.fr}



\begin{document}

\begin{abstract}
We are interested in viscous scalar conservation laws with a white-in-time but spatially correlated stochastic forcing. The equation is assumed to be one-dimensional and periodic in the space variable, and its flux function to be locally Lipschitz continuous and have at most polynomial growth. Neither the flux nor the noise need to be non-degenerate. In a first part, we show the existence and uniqueness of a global solution in a strong sense. In a second part, we establish the existence and uniqueness of an invariant measure for this strong solution.
\end{abstract}

\maketitle

\section{Introduction}

\subsection{Stochastic viscous scalar conservation law}

We are interested in the existence, uniqueness, regularity and large time behaviour of solutions of the following viscous scalar conservation law with additive and time-independent stochastic forcing
\begin{equation}\label{SSCL}
\dd u = - \partial_x A(u) \dd t + \nu \partial_{xx} u \dd t + \sum_{k\geq1} g_k \dd W^k(t) , \quad x \in \T , \quad t \geq 0,
\end{equation}
where $(W^k(t))_{t \geq 0}$, $k \geq 1$, is a family of independent Brownian motions. Here, $\T$ denotes the one-dimensional torus $\R/\Z$, meaning that the sought solution is periodic in space. The \textit{flux function} $A$ is assumed to satisfy the following set of conditions.
\begin{assumption}[on the flux function]\label{ass:A}
  The function $A : \R \to \R$ is $C^2$ on $\R$, its first derivative has at most polynomial growth:
    \begin{equation}\label{PGA}
      \exists C_1>0 , \quad \exists p_A \in \N^* , \quad \forall v \in \R , \qquad | A'(v) | \leq C_1 \left( 1+ |v|^{p_A} \right),
    \end{equation}
    and its second derivative $A''$ is locally Lipschitz continuous on $\R$.
\end{assumption}
The parameter $\nu>0$ is the \textit{viscosity} coefficient. In order to present our assumptions on the family of functions $g_k : \T \to \R$, $k \geq 1$, which describe the spatial correlation of the stochastic forcing of~\eqref{SSCL}, we first introduce some notation. For any $p\in[1,+\infty]$, we denote by $L^p_0(\T)$ the subset of functions $v \in L^p(\T)$ such that
\[
  \int_\T v \dd x = 0.
\]
The $L^p$ norm induced on $L^p_0(\T)$ is denoted by $\|\cdot\|_{L^p_0(\T)}$. For any integer $m \geq 0$, we denote by $H^m_0(\T)$ the intersection of the Sobolev space $H^m(\T)$ with $L^2_0(\T)$. Equipped with the norm
\[
  \|v\|_{H^m_0(\T)} := \left( \int_{\T} | \partial_x^m v |^2 \dd x \right)^{1/2}, 
\]
and the associated scalar product $\langle \cdot, \cdot\rangle_{H^m_0(\T)}$, it is a separable Hilbert space. On the one-dimensional torus, the Poincar\'e inequality implies that $H^{m+1}_0(\T) \subset H^m_0(\T)$ and $\|\cdot\|_{H^m_0(\T)} \leq \|\cdot\|_{H^{m+1}_0(\T)}$. Actually, the following stronger inequality holds: if $v\in H^1_0(\T)$, then $v\in L^\infty_0(\T)$ and for all $p \in [1,+\infty)$,
\begin{equation}\label{strongerpoincare}
  \|v\|_{L^p_0(\T)}\leq \|v\|_{L^\infty_0(\T)}\leq\|v\|_{H^1_0(\T)}.
\end{equation}

The spaces $H^m_0(\T),m\geq0,$ generalise to the class of fractional Sobolev spaces $H^s_0(\T)$, where $s\in[0,+\infty)$, which will be defined in Section \ref{sec:mild}.
We may now state:
\begin{assumption}[on the noise functions]\label{ass:g}
  For all $k \geq 1$, $g_k \in H^2_0(\T)$ and
  \begin{equation}\label{boundgk}
    D_0 := \sum_{k\geq1} \| g_k \|_{H^2_0(\T)}^2 < +\infty.
  \end{equation}
\end{assumption}

Let $(\Omega,\mathcal{F},\P)$ be a probability space, equipped with a normal filtration $(\mathcal{F}_t)_{t \geq 0}$ in the sense of~\cite[Section~3.3]{DZ92}, on which $(W^k)_{k \geq 1}$ is a family of independent Brownian motions. Under Assumption~\ref{ass:g}, the series $\sum_k g_k W^k$ converges in $L^2(\Omega,C([0,T],H^2_0(\T)))$, for any $T>0$, towards an $H^2_0(\T)$-valued Wiener process $(W^Q(t))_{t\in[0,T]}$ with respect to the filtration $(\mathcal{F}_t)_{t \geq 0}$, defined in the sense of \cite[Section 4.2]{DZ92}, with the trace class covariance operator $Q : H^2_0(\T) \to H^2_0(\T)$ given by
\begin{equation}\label{Q}
\forall u,v \in H^2_0(\T), \qquad \langle u, Qv\rangle_{H^2_0(\T)} = \sum_{k \geq 1} \langle u, g_k\rangle_{H^2_0(\T)}\langle v, g_k\rangle_{H^2_0(\T)}.
\end{equation}
Thus, almost surely, $t\mapsto W^Q(t)$ is continuous in $H^2_0(\T)$ and for all $u \in H^2_0(\T)$, the process $( \langle W^Q(t) , u \rangle_{H^2_0(\T)} )_{t\geq0}$ is a real-valued Wiener process with variance
\begin{equation}\label{covariance}
\E \left[\left\langle W^Q(t),u\right\rangle_{H^2_0(\T)}^2\right]=t\sum_{k\geq1}\langle g_k,u\rangle_{H^2_0(\T)}^2 .
\end{equation}

\subsection{Main results and previous works}\label{ss:mainresults}

First, we are interested in the well-posedness in the strong sense of Equation \eqref{SSCL}. In particular, we look for solutions that admit at least a second spatial derivative in order to give a classical meaning to the viscous term, in the sense of the following definition:
\begin{definition}[Strong solution to~\eqref{SSCL}]\label{defi:strong}
  Let $u_0 \in H^2_0(\T)$. Under Assumptions~\ref{ass:A} and~\ref{ass:g}, a strong solution to Equation~\eqref{SSCL} with initial condition $u_0$ is an $(\mathcal{F}_t)_{t \geq 0}$-adapted process $(u(t))_{t \geq 0}$ with values in $H^2_0(\T)$ such that, almost surely:
  \begin{enumerate}
    \item the mapping $t \mapsto u(t)$ is continuous from $[0,+\infty)$ to $H^2_0(\T)$;
    \item for all $t \geq 0$, the following equality holds:
    \begin{equation}\label{strongsolution}
    u(t) = u_0 + \int_0^t \left( - \partial_x A \left( u(s) \right) + \nu \partial_{xx} u(s) \right) \dd s + W^Q(t) .
    \end{equation} 
  \end{enumerate}
\end{definition}

In the above definition, the first condition ensures that the time integral in Equation \eqref{strongsolution} is a well-defined Bochner integral in $L^2_0(\T)$. For a careful introduction of the general concepts of random variables and stochastic processes in Hilbert spaces, the reader is referred to the third and fourth chapters of the reference book \cite{DZ92}.

Our first result is the following:
\begin{theorem}[Well-posedness]\label{wellposedness}
Let $u_0 \in H^2_0(\T)$. Under Assumptions \ref{ass:A} and \ref{ass:g}, there exists a unique strong solution $(u(t))_{t \geq 0}$ to Equation~\eqref{SSCL} with initial condition $u_0$. Moreover, the solution depends continuously on initial data in the following sense: if $( u_0^{(j)} )_{j\geq1}$ is a sequence of $H^2_0(\T)$ satisfying
\[ \lim_{j\to\infty} \left\| u_0 - u^{(j)}_0 \right\|_{H^2_0(\T)} = 0 , \]
then, denoting by $( u^{(j)}(t) )_{t\geq0,j\geq1}$ the family of associated solutions, for any $T\geq0$, we have almost surely
\[ \lim_{j\to\infty} \sup_{t\in[0,T]} \left\| u(t) - u^{(j)}(t) \right\|_{H^2_0(\T)} = 0 . \]
\end{theorem}

Similar results have already been established: the case where the flux $A$ is strictly convex is treated in \cite[Appendix A]{Bor12b}, and the case where $A$ is globally Lipschitz continuous {is treated in \cite{Hof13}}. Furthermore, the case of \textit{mild} solutions (in $L^p$ spaces) has been looked at in \cite{GR00}. Here, no global Lipschitz continuity assumption nor restrictions on the convexity of the flux function are made. We can also point out that the well-posedness of stochastically forced conservations laws in the inviscid case (\textit{i.e.} when $\nu=0$) has been under a great deal of investigation in the recent years. In this "hyperbolic" framework, the appearance of shocks prevents the solutions to be smooth enough to be considered in a strong sense as in our present work. Therefore, the study of \textit{entropic} solutions \cite{FN08} or \textit{kinetic} solutions \cite{DV14,GH18} to the SPDE have been the two main approaches, both of which rely on a \textit{vanishing viscosity} argument: the entropic or kinetic solution is sought as the limit of its viscous approximation as the viscosity coefficient tends to $0$.

More recent works concern the Burgers equation with \textit{stochastic transport noise} in the viscous and inviscid cases~\cite{ABT19}, or the spatial regularity for solutions of the viscous Burgers equation with additive noise~\cite{JLP19}. A natural extension of our works would be to consider a viscous conservation law with multiplicative noise or even, as in~\cite{ABT19}, a transport noise.

Let $C_b(H^2_0(\T))$ denote the set of continuous and bounded functions from $H^2_0(\T)$ to $\R$. As a consequence of Theorem \ref{wellposedness}, we can define a family of functionals $(P_t)_{t\geq0}$ on $C_b(H^2_0(\T))$ by writing
\[ P_t\varphi(u_0) := \E_{u_0} \left[ \varphi(u(t)) \right] , \qquad t\geq0 , \quad u_0 \in H^2_0(\T) , \]
where the notation $\E_{u_0}$ indicates that the random variable $u(t)$ is the solution to \eqref{SSCL} at time $t$ starting from the initial condition $u_0$.

\begin{corollary}\label{markov}
 Under Assumptions \ref{ass:A} and \ref{ass:g}, the family $(P_t)_{t\geq0}$ is a Feller semigroup and the process $(u(t))_{t\geq0}$ is a strong Markov process in $H^2_0(\T)$ with semigroup $(P_t)_{t\geq0}$.
\end{corollary}
\begin{proof}
The uniqueness of a strong solution and the fact that, for all $t\geq0$, the processes $( W^Q(t+s)-W^Q(t))_{s\geq0}$ and $(W^Q(s))_{s\geq0}$ have the same distribution, ensure that $(P_t)_{t\geq0}$ is a semigroup, and therefore that $(u(t))_{t\geq0}$ is a Markov process. The Feller property is a straightforward consequence of the result of continuous dependence on initial conditions given in Theorem \ref{wellposedness}, whereas it is a classical result that the strong Markov property of $(u(t))_{t\geq0}$ follows from the Feller property of $(P_t)_{t\geq0}$ (see for instance the proof of \cite[Theorem 16.21]{Bre68}).
\end{proof}

Let $\mathcal{B}(H^2_0(\T))$ denote the Borel $\sigma$-algebra of the metric space $H^2_0(\T)$, and $\mathcal{P}(H^2_0(\T))$ refer to the set of Borel probability measures on $H^2_0(\T)$. The Markov property allows us to extend the notion of strong solution to \eqref{SSCL} by considering not only a deterministic initial condition but any $\mathcal F_0$-measurable random variable $u_0$ on $H^2_0(\T)$. In this perspective, we define the dual semigroup $\left(P^*_t\right)_{t\geq0}$ of $(P_t)_{t\geq0}$ by
\[ P^*_t\alpha(\Gamma) := \int_{H^2_0(\T)} \P_{u_0} \left( u(t) \in \Gamma \right) \dd\alpha(u_0) , \qquad t\geq0,\quad \alpha \in \mathcal P\left(H^2_0(\T)\right),\quad \Gamma \in \mathcal B\left(H^2_0(\T)\right). \]
In particular, $P_t^*\alpha$ is the law of $u(t)$ when $u_0$ is distributed according to $\alpha$.

\begin{definition}[Invariant measure]\label{defi:IM}
We say that a probability measure $\mu \in \mathcal P (H^2_0(\T))$ is an invariant measure for the semigroup $(P_t)_{t\geq0}$ (or equivalently for the process $\left( u(t) \right)_{t\geq0}$) if and only if
\[ \forall t \geq 0 , \quad P^*_t \mu = \mu . \]
\end{definition}

\begin{theorem}[Existence, uniqueness and estimates on the invariant measure]\label{thm:im}
Under Assumptions \ref{ass:A} and \ref{ass:g}, the process $(u(t))_{t\geq0}$ solution to the SPDE \eqref{SSCL} admits a unique invariant measure $\mu$. Besides, if $u \in H^2_0(\T)$ is distributed according to $\mu$, then $\E[\|u\|^2_{H^2_0(\T)}] < +\infty$ and, for all $p \in [1,+\infty)$, $\E[\|u\|^p_{L^p_0(\T)}] < +\infty$.
\end{theorem}

A few similar results exist in the literature. Da Prato, Debussche and Temam \cite{DDT94} have studied the viscous Burgers equation (which corresponds to the flux function $A(u)=u^2/2$) perturbed by an additive space-time white noise whereas Da Prato and Gatarek \cite{DG95} studied the same equation but with a multiplicative white noise. Both showed the well-posedness of the equation as well as the existence of an invariant measure. These results are moreover put in a much detailed context in the two reference books \cite{DZ92,DZ96}. Boritchev \cite{Bor12,Bor12b,Bor13} showed the existence and uniqueness of an invariant measure for the viscous \textit{generalised} Burgers equation (which corresponds to the case of strictly convex flux function) perturbed by a white-in-time and spatially correlated noise. E, Khanin, Mazel and Sinai \cite{EKMS00} showed the existence and uniqueness of an invariant measure for the inviscid Burgers equation with a white-in-time and spatially correlated noise. Debussche and Vovelle \cite{DV15} generalised this last result by extending it to \textit{non-degenerate} flux functions (roughly speaking, there is no non-negligible subset of $\R$ on which $A$ is linear). Besides, the fact that these results from \cite{EKMS00,DV15} also hold when $\nu=0$ makes them quite powerful: it shows indeed that the presence of a viscous term is not a necessary condition for the solution to be stationary. On this topic, we refer the reader to a recent nicely detailed survey by Chen and Pang~\cite{CP19}.

The stochastic Burgers equation is mainly studied as a one-dimensional model for \textit{turbulence}. By showing a stable behaviour at large times, this model manages, to some extent, to fit the predicitions of Kolmogorov's "K41" theory about the \textit{universal} properties of a turbulent flow \cite{Kol41a,Kol41b}. Whether it is modelled by the Burgers equation or a by more general process such as Equation \eqref{SSCL}, turbulence is then described through the statistics of some particular small-scale quantities in the stationary state \cite{E00,E01}. Sharp estimates were given by Boritchev for these small-scale quantities \cite{Bor12b}, which were furthermore shown to be independent of the viscosity coefficient. One of the purposes of this paper is to lay the groundwork for the numerical analysis of Equation \eqref{SSCL}. In a companion paper \cite{BMR19}, we introduce a finite-volume approximation of \eqref{SSCL} which allows to approximate the invariant measure $\mu$. Generating random variables with distribution $\mu$ shall eventually lead us to compute said small-scale quantities and analyse the development of turbulence in the model established by Equation \eqref{SSCL}.

\subsection{Outline of the article}
The proofs of Theorems \ref{wellposedness} and \ref{thm:im} are respectively detailed in Sections \ref{XP} and \ref{IM}.

\section{Well-posedness and regularity}\label{XP}

This section is dedicated to the proof of Theorem \ref{wellposedness}. This proof is decomposed as follows. In Subsection \ref{sec:mild}, we introduce a weaker formulation of Equation \eqref{SSCL}, the so-called \textit{mild formulation}. In Subsection~\ref{LWP}, we show that Equation~\eqref{SSCL} is well-posed locally in time both in the mild and in the strong sense. In Subsection~\ref{sec:estimates}, we give higher bounds for the Lebesgue and Sobolev norms of this local solution. Eventually, these estimates allow us to extend the local solution to a global-in-time solution, and thus to prove Theorem~\ref{wellposedness} in Subsection~\ref{proofthm1}. In the sequel, some results (Propositions~\ref{localboundedness}, \ref{prop:heatkernel}, \ref{continuity} and \ref{prop:mild-strong}) are either standard or mild adaptations of results which are proved elsewhere. We omit their proof here and refer to Subsection~2.2.5 in~\cite{Mar19} for details.

\subsection{Mild formulation of \eqref{SSCL}}\label{sec:mild} 
In this subsection, we collect preliminary results which shall enable us to provide a \textit{mild} formulation of Equation \eqref{SSCL}, for which we prove the existence and uniqueness of a solution on a small interval.

\subsubsection{Fractional Sobolev spaces}

For all $m' \geq 1$, let us define $\lambda_{2m'-1} = \lambda_{2m'} = -(2\pi m')^2$, and $e_{2m'-1}(x) = \sqrt{2}\sin(2\pi m'x)$, $e_{2m'}(x) = \sqrt{2}\cos(2\pi m'x)$. The family $(e_m)_{m \geq 1}$ is a complete orthogonal basis of $L^2_0(\T)$ such that, for all $m \geq 1$, $e_m$ is $C^\infty$ on $\T$ and $\partial_{xx} e_m = \lambda_m e_m$. With respect to this basis, we define the fractional Sobolev space $H^s_0(\T)$, for any $s\in[0,+\infty)$, as the space of functions $v\in L^2_0(\T)$ such that
\begin{equation}\label{fractional}
\| v \|_{H^s_0(\T)} := \left( \sum_{m\geq1} (-\lambda_m)^s \langle v,e_m \rangle_{L^2_0(\T)}^2 \right)^{1/2} < +\infty.
\end{equation}

We take from \cite[Appendice A]{Bor12b} the following proposition and adapt it to our case of a flux function satisfying Assumption \ref{ass:A}:
\begin{prop}\label{localboundedness}
Under Assumption \ref{ass:A}, for any $s\in[1,2]$, the mapping
\[ v \in H^s_0(\T) \longmapsto \partial_x A(v) \in H^{s-1}_0(\T) \]
is bounded on bounded subsets of $H^s_0(\T)$. Moreover, when $s=1$ or $s=2$, it is Lipschitz continuous on bounded subsets of $H^s_0(\T)$.
\end{prop}

By virtue of Proposition \ref{localboundedness}, for all $m \geq 1$, we denote by $C_2^{(m)}$ and $C_3^{(m)}$ two finite constants such that:
\begin{itemize}
  \item for all $v \in H^1_0(\T)$ such that $\|v\|_{H^1_0(\T)} \leq m$, $\|\partial_x A(v)\|_{L^2_0(\T)} \leq C_2^{(m)}$;
  \item for all $v_1, v_2 \in H^1_0(\T)$ such that $\|v_1\|_{H^1_0(\T)} \vee \|v_2\|_{H^1_0(\T)} \leq m$, $\|\partial_x A(v_1)-\partial_x A(v_2)\|_{L^2_0(\T)} \leq C_3^{(m)}\|v_1-v_2\|_{H^1_0(\T)}$.
\end{itemize}

\subsubsection{Heat kernel}

Let us denote by $(S_t)_{t\geq0}$ the semigroup generated by the operator $\nu \partial_{xx}$:
\begin{equation}\label{semigroupdef}
S_t v := \sum_{m\geq1} \e^{\nu\lambda_m t} \langle v ,e_m\rangle_{L^2_0(\T)} e_m , \quad v\in L^2_0(\T), \quad t\geq0 .
\end{equation}
Some of its properties are gathered in the following proposition.

\begin{prop}[Properties of the heat kernel]\label{prop:heatkernel}
  The semigroup $(S_t)_{t \geq 0}$ satisfies the following properties.
  \begin{enumerate}
    \item For any $s \geq 0$, for any $v \in H^s_0(\T)$, for any $t \geq 0$, $S_t v \in H^s_0(\T)$ and $\|S_t v\|_{H^s_0(\T)} \leq \|v\|_{H^s_0(\T)}$; besides, the mapping $t \mapsto S_tv \in H^s_0(\T)$ is continuous on $[0,+\infty)$.
    \item For all $0 \leq s_1 \leq s_2$, there exists a constant $C_4=C_4(s_1,s_2)>0$ such that
\[ \forall v \in H^{s_1}_0(\T) , \quad \forall t \geq 0 , \qquad \| S_t v \|_{H^{s_2}_0(\T)} \leq C_4 t^{\frac{s_1-s_2}{2}} \| v \|_{H^{s_1}_0(\T)} . \]
    \item For any $s\in[0,+\infty)$, $T>0$ and $(v(t))_{t\in[0,T]} \in C( [0,T],H^s_0 (\T))$, the process $( \int_0^t S_{t-r}v(r)\dd r)_{t\in[0,T]}$ belongs to $C( [0,T],H^{s+3/2}_0(\T))$.
  \end{enumerate}
\end{prop}

\subsubsection{Stochastic convolution and mild formulation of \eqref{SSCL}}

Let $(\bar{\mathcal{F}}_t)_{t \geq 0}$ be a normal filtration on the probability space $(\Omega, \mathcal{F}, \P)$ and $(\bar{W}^Q(t))_{t \geq 0}$ be a $Q$-Wiener process in $H^2_0(\T)$ with respect to this filtration. Given that the orthonormal basis $(e_m)_{m\geq1}$ of the space $L^2_0(\T)$ satisfies $\partial_{xx} e_m = \lambda_m e_m$, the family $(e_m/\lambda_m)_{m\geq1}$ is an orthonormal basis of $H^2_0(\T)$.
We set
\[ \bar{W}_m(t) := \left\langle \bar{W}^Q(t) , \frac{e_m}{\lambda_m}\right\rangle_{H^2_0(\T)} , \qquad  m\geq1 , \quad t\geq0 , \]
so that by \eqref{covariance}, $(\bar{W}_m(t))_{t\geq0}$ is a real-valued Brownian motion with variance $\sum_{k \geq 1} \langle g_k, e_m/\lambda_m\rangle_{H^2_0(\T)}^2$. Next, we write
\[ \bar{w}_m(t) := \int_0^t \e^{\nu\lambda_m(t-s)} \dd \bar{W}_m(s) , \qquad m\geq1, \quad t\geq0 . \]

\begin{prop}\label{continuity}
Under Assumption \ref{ass:g}, for all $T>0$, the series \[ \sum_{m\geq1} \frac{e_m}{\lambda_m}(\bar{w}_m(t))_{t \in [0,T]} \] converges in $L^2 ( \Omega , C( [0,T],H^2_0(\T)))$, and its sum defines an $(\bar{\mathcal{F}}_t)_{t\geq0}$-adapted, $H^2_0(\T)$-valued process $(\bar{w}(t))_{t \geq 0}$ almost surely continuous.
\end{prop}
The process $(\bar{w}(t))_{t \geq 0}$ is called the \textit{stochastic convolution} associated to the $Q$-Wiener process $(\bar{W}^Q(t))_{t \geq 0}$.

In the sequel, we let $\bar{\tau}$ be a $(\bar{\mathcal{F}}_t)_{t \geq 0}$-stopping time, almost surely finite. We shall say that a process $(\bar{u}(t))_{t \in [0,\bar{\tau}]}$ is $(\bar{\mathcal{F}}_t)_{t \geq 0}$-adapted if for all $t \geq 0$, the random variable $\bar{u}(t)\mathbf{1}_{t \leq \bar{\tau}}$ is $\bar{\mathcal{F}}_t$-measurable.

\begin{definition}[Local mild solution]\label{defi:sol}
  Let $\bar{u}_0$ be an $\bar{\mathcal{F}}_0$-measurable, $H^1_0(\T)$-valued random variable. Under Assumptions \ref{ass:A} and \ref{ass:g}, a (local) mild solution to the SPDE
  \begin{equation}\label{eq:edpsbar}
    \dd \bar{u}(t) = -\partial_x A(\bar{u}(t))\dd t + \nu \partial_{xx} \bar{u}(t)\dd t + \dd \bar{W}^Q(t)
  \end{equation}
  on $[0,\bar{\tau}]$ is an $H^1_0(\T)$-valued, $(\bar{\mathcal{F}}_t)_{t \geq 0}$-adapted process $(\bar{u}(t))_{t \in [0,\bar{\tau}]}$ such that, almost surely:
  \begin{enumerate}
    \item the mapping $t \mapsto \bar{u}(t) \in H^1_0(\T)$ is continuous on $[0,\bar{\tau}]$;
    \item for all $t \in [0,\bar{\tau}]$,
    \begin{equation}\label{solution}
      \bar{u}(t) = S_t \bar{u}_0 - \int_0^t S_{t-s} \partial_x A(\bar{u}(s)) \dd s + \bar{w}(t) .
    \end{equation} 
  \end{enumerate}
\end{definition}
The combination of Propositions \ref{localboundedness} and \ref{prop:heatkernel} ensures that all terms of the identity \eqref{solution} are well-defined.

We now clarify the relationship between the notions of mild and strong solutions.

\begin{prop}[Mild and strong solutions]\label{prop:mild-strong}
  Under the assumptions of Definition \ref{defi:sol}, let $(\bar{u}(t))_{t \in [0,\bar{\tau}]}$ be a mild solution to \eqref{eq:edpsbar} on $[0,\bar{\tau}]$. If $\bar{u}_0 \in H^2_0(\T)$, then:
  \begin{enumerate}
    \item for all $t \in [0,\bar{\tau}]$, $\bar{u}(t) \in H^2_0(\T)$ and the mapping $t \mapsto \bar{u}(t) \in H^2_0(\T)$ is continuous on $[0,\bar{\tau}]$;
    \item for all $t \in [0,\bar{\tau}]$,
  \begin{equation*}
    \bar{u}(t) = \bar{u}_0 + \int_0^t \left( - \partial_x A \left( \bar{u}(s) \right) + \nu \partial_{xx} \bar{u}(s) \right) \dd s + \bar{W}^Q(t).
  \end{equation*}
  \end{enumerate}
  Conversely, any $H^2_0(\T)$-valued, $(\bar{\mathcal{F}}_t)_{t \geq 0}$-adapted process $(\bar{u}(t))_{t \in [0,\bar{\tau}]}$ satisfying these two conditions almost surely is a mild solution to \eqref{eq:edpsbar} on $[0,\bar{\tau}]$.
\end{prop}

\subsubsection{Existence and uniqueness of a mild solution on a small interval} For any integer $\bar{m}_0 \geq 0$, let us define
\begin{equation*}
  \tau_{\bar{m}_0}\left(\bar{W}^Q\right) = \frac{1}{8\left(C_1C_3^{(\bar{m}_0+1)}\right)^2} \wedge \inf\left\{t \geq 0: 2C_4C_2^{(\bar{m}_0+1)}\sqrt{t} + \|\bar{w}(t)\|_{H^1_0(\T)} \geq 1\right\},
\end{equation*}
where we recall that the constant $C_4$ is defined in Proposition \ref{prop:heatkernel}, the constants $C_2^{(m)}$ and $C_3^{(m)}$ are defined after Proposition \ref{localboundedness}, and the constant $C_1$ is defined in \eqref{PGA}.

Notice that $\tau_{\bar{m}_0}(\bar{W}^Q) \in (0,+\infty)$, almost surely.

In the spirit of \cite{DDT94,Bor12b}, we obtain the existence and uniqueness of a mild solution to \eqref{eq:edpsbar} on the "small" interval $[0,\tau_{\bar{m}_0}(\bar{W}^Q)]$ by a fixed-point argument.

\begin{lemma}[Local existence and uniqueness]\label{localWP}
Let $\bar{u}_0$ and $\bar{m}_0$ be two $\bar{\mathcal{F}}_0$-measurable random variables taking values respectively in $H^1_0(\T)$ and $\N$ such that $\| \bar{u}_0 \|_{H^1_0(\T)}\leq \bar{m}_0$. Furthermore, let us set $\bar{\tau} := \tau_{\bar{m}_0}(\bar{W}^Q)$. Then, under Assumptions \ref{ass:A} and \ref{ass:g}, there is a unique mild solution $(\bar{u}(t))_{t \in [0,\bar{\tau}]}$ to \eqref{eq:edpsbar} on $[0,\bar{\tau}]$.
\end{lemma}
\begin{proof}
Let us introduce the random set
\[ \Sigma := \left\{ \left(v(t)\right)_{t\in[0,\bar{\tau}]} \in C\left([0,\bar{\tau}],H^1_0(\T)\right) : \forall t \in [0,\bar{\tau}], \|v(t)\|_{H^1_0(\T)} \leq \bar{m}_0+ 1\right\}. \]
Thanks to Propositions \ref{prop:heatkernel} and \ref{continuity}, we may define the random operator $G : C([0,\bar{\tau}],H^1_0(\T)) \to C([0,\bar{\tau}],H^1_0(\T))$ by
\[ (Gv)(t) = S_t \bar{u}_0 - \int_0^t S_{t-s} \partial_x A(v(s))\dd s + \bar{w}(t) , \qquad t \in[0,\bar{\tau}] , \]
and notice that any $v \in C([0,\bar{\tau}],H^1_0(\T))$ satisfies Equation \eqref{solution} if and only if $Gv=v$.

We first write, for some $v \in C([0,\bar{\tau}],H^1_0(\T))$ and for any $t\in[0,\bar{\tau}]$,
\begin{equation}
\|(Gv)(t)\|_{H^1_0(\T)} \leq \|S_t \bar{u}_0\|_{H^1_0(\T)} + \int_0^t \left\|S_{t-s} \partial_x A(v(s))\right\|_{H^1_0(\T)}\dd s + \| \bar{w}(t) \|_{H^1_0(\T)} .
\end{equation}
  
On the one hand, by the first assertion of Proposition \ref{prop:heatkernel}, $\|S_t \bar{u}_0\|_{H^1_0(\T)} \leq \| \bar{u}_0\|_{H^1_0(\T)} \leq \bar{m}_0$; on the other hand, we know thanks to the second assertion of Proposition \ref{prop:heatkernel} that
\begin{equation}
\left\|S_{t-s} \partial_x A(v(s))\right\|_{H^1_0(\T)} \leq \frac{C_4}{\sqrt{t-s}}\|\partial_x A(v(s))\|_{L^2_0(\T)} ,
\end{equation}
furthermore, thanks to Proposition \ref{localboundedness}, if $v \in \Sigma$, then $\partial_x A(v)$ is bounded in $L^2_0(\T)$ uniformly in time, \textit{i.e.} for all $s \in [0,\bar{\tau}]$, $\| \partial_xA(v(s))\|_{L^2_0(\T)} \leq C_2^{(\bar{m}_0+1)}$. Thus,
\begin{equation}
\|(Gv)(t)\|_{H^1_0(\T)} \leq \bar{m}_0 + 2C_4C_2^{(\bar{m}_0+1)} \sqrt{t} + \| \bar{w}(t) \|_{H^1_0(\T)} , \qquad t\in[0,\bar{\tau}].
\end{equation}
By definition of $\bar{\tau}$, it follows that $Gv \in \Sigma$ whenever $v \in \Sigma$.

We now take $(v_1(t))_{t\in\left[0,\bar{\tau}\right]},(v_2(t))_{t\in\left[0,\bar{\tau}\right]} \in \Sigma$. Then, for any $t \in \left[0,\bar{\tau}\right]$,
\begin{equation}
\begin{aligned}
\|(Gv_1)(t)-(Gv_2)(t)\|_{H^1_0(\T)} &= \left\|\int_0^t S_{t-s}\left(\partial_x A(v_1(s))-\partial_x A(v_2(s))\right)\dd s\right\|_{H^1_0(\T)}\\
& \leq \int_0^t \frac{C_4}{\sqrt{t-s}} \|\partial_xA(v_1(s))-\partial_xA(v_2(s))\|_{L^2_0(\T)}\dd s,
\end{aligned}
\end{equation}
where we have used the same arguments as above. Using now the Lipschitz continuity result in Proposition \ref{localboundedness} and the definition of $\bar{\tau}$, we get for all $t \in \left[0,\bar{\tau}\right]$,
\begin{align*}
\|(Gv_1)(t)-(Gv_2)(t)\|_{H^1_0(\T)} &\leq 2C_1C_3^{(\bar{m}_0+1)} \sqrt{t} \sup_{s \in [0,t]} \| v_1(s)-v_2(s) \|_{H^1_0(\T)} \\
&\leq \frac12 \sup_{s\in[0,\bar{\tau}]} \| v_1(s)-v_2(s) \|_{H^1_0(\T)} ,
\end{align*}
meaning that $G$ is a contraction mapping on $\Sigma$, which is complete. Then, by the Banach fixed-point theorem, $G$ admits a unique fixed point $(\bar{u}(t))_{t\in\left[0,\bar{\tau}\right]}$ in $\Sigma$. To show that this solution to Equation \eqref{solution} is unique among all the $H^1_0(\T)$-valued continuous processes, let us first notice that our choice of $\bar{\tau}$ implies
\[ \forall t < \bar{\tau}, \qquad \|\bar{u}(t)\|_{H^1_0(\T)} < \bar{m}_0+1.\]
Assume that there is another solution $(\tilde u(t))_{t\in\left[0,\bar{\tau}\right]}$ of \eqref{solution} not belonging almost surely to $\Sigma$. Then we have with positive probability
\[ \exists \tilde{\tau} < \bar{\tau}, \qquad \|\tilde{u}(\tilde\tau)\|_{H^1_0(\T)} \geq \bar{m}_0+1.\]
This means that the double inequality $\|\bar{u}\left(\tilde{\tau}\right)\|_{H^1_0(\T)} < \bar{m}_0+1 \leq \|\tilde u\left(\tilde{\tau}\right)\|_{H^1_0(\T)}$ holds on some non-negligible event. On this event, the fixed-point argument also holds in the set
\[ \tilde\Sigma := \left\{ (v(t))_{t\in[0,\tilde{\tau}]} : \forall t \in [0,\tilde{\tau}] , \|v(t)\|_{H^1_0(\T)} \leq \bar{m}_0+1 \right\} \]
which is formally a subset of $\Sigma$. Thus, by uniqueness of the fixed point, we have $\bar{u}_{|\left[0,\tilde{\tau}\right]} = \tilde u_{|\left[0,\tilde{\tau}\right]}$ and in particular $\bar{u}\left(\tilde{\tau}\right)=\tilde u\left(\tilde{\tau}\right)$, which is absurd. As a consequence, $(\bar{u}(t))_{t\in\left[0,\bar{\tau}\right]}$ is the only $H^1_0(\T)$-valued process with continuous trajectories satisfying Equation \eqref{solution} on $[0,\bar{\tau}]$.

Finally, let $v^{(0)} = 0$ and define the sequence of processes $v^{(j)} \in C([0,\bar{\tau}],H^1_0(\T))$, $j \geq 1$ by $v^{(j)} = Gv^{(j-1)}$. It is clear from the definition of the operator $G$ and from Proposition \ref{continuity} that each process $( v^{(j)}(t) \ind{t\leq\bar{\tau}} )_{t\geq0}$ is $(\bar{\mathcal{F}}_t)_{t \geq0}$-adapted. On the other hand, the Banach fixed-point theorem asserts that almost surely, the sequence $(v^{(j)}(t))_{t \in [0,\bar{\tau}]}$ converges to $(\bar{u}(t))_{t \in [0, \bar{\tau}]}$ in $C([0,\bar{\tau}],H^1_0(\T))$. As a consequence, for any $t \geq 0$, the sequence of $\bar{\mathcal{F}}_t$-measurable random variables $\mathbf{1}_{t \leq \bar{\tau}} v^{(j)}(t)$ converges almost surely to $\mathbf{1}_{t \leq \bar{\tau}} \bar u(t)$, which makes this limit also $\bar{\mathcal{F}}_t$-measurable. Thus, the process $(\mathbf{1}_{t \leq \bar{\tau}} \bar{u}(t))_{t \geq0}$ is $(\bar{\mathcal{F}}_t)_{t \geq 0}$-adapted.
\end{proof}

\subsection{Construction of a maximal solution to \eqref{SSCL}}\label{LWP} In this subsection, we use the notions introduced in Subsection \ref{sec:mild} to prove the following existence and uniqueness result for \eqref{SSCL}.

\begin{lemma}[Existence and uniqueness result of a maximal solution to \eqref{SSCL}]\label{lem:solmax}
  Under Assumptions \ref{ass:A} and \ref{ass:g}, for any $u_0 \in H^1_0(\T)$, there exists a pair $(T^*, (u(t))_{t \in [0, T^*)})$ such that:
  \begin{enumerate}
    \item for any $(\mathcal{F}_t)_{t \geq 0}$-stopping time $T$ such that almost surely, $T < +\infty$ and $T \leq T^*$, $(u(t))_{t \in [0,T]}$ is the unique mild solution to \eqref{SSCL} on $[0,T]$;
    \item almost surely, $T^*=+\infty$ or $\limsup_{t \to T^*} \|u(t)\|_{H^1_0(\T)} = +\infty$.
  \end{enumerate}
\end{lemma}
The random time $T^*$ is called the \textit{explosion time} and the process $(u(t))_{t \in [0,T^*)}$ is called the \textit{maximal solution} to \eqref{SSCL}.

\begin{proof}
Let $u_0 \in H^1_0(\T)$. Let $m^{(0)}_0 = \lceil \|u_0\|_{H^1_0(\T)} \rceil$. By Lemma \ref{localWP}, Equation \eqref{SSCL} possesses a unique mild solution $(u(t))_{t \in [0,\tau^{(0)}]}$ on $[0,\tau^{(0)}]$, where $\tau^{(0)} = \tau_{m^{(0)}_0}$. We now define the filtration $(\mathcal{F}^{(1)}_t)_{t \geq 0}$ by 
\[ \mathcal{F}^{(1)}_t = \mathcal F_{\tau^{(0)}+t} = \left\{ B\in\mathcal F : \forall s\geq0, B\cap \{ \tau^{(0)}+t \leq s \} \in \mathcal F_s \right\}, \]
and recall that the process $W^{Q,(1)}$ defined by $W^{Q,(1)}(t) = W^Q(\tau^{(0)}+t) - W^Q(t)$ is a $Q$-Wiener process with respect to $(\mathcal{F}^{(1)}_t)_{t \geq 0}$. Therefore, applying Lemma \ref{localWP} again with this $Q$-Wiener process, and initial condition $u_0^{(1)} = u(\tau^{(0)})$ and $m^{(1)}_0 = \lceil \| u(\tau^{(0)}) \|_{H^1_0(\T)} \rceil \vee m^{(0)}_0$, we obtain a mild solution $(u^{(1)}(t))_{t \in [0,\tau^{(1)}]}$ of $\dd u = -\partial_x A(u)\dd t + \nu \partial_{xx} u \dd t + \dd W^{Q,(1)}$ on $[0,\tau^{(1)}]$, where $\tau^{(1)} = \tau_{m^{(1)}_0}(W^{Q,(1)})$. It is then easily checked that defining $T^{(1)} = \tau^{(0)}+\tau^{(1)}$ and $u(t+\tau^{(0)}) = u^{(1)}(t)$ for any $t \in (0,\tau^{(1)}]$, we obtain a unique mild solution $(u(t))_{t \in [0, T^{(1)}]}$ to Equation \eqref{SSCL} on $[0,T^{(1)}]$.

We now proceed by induction and set for all $n\geq 1$,
\begin{align*}
T^{(n)} &:= \sum_{i=0}^n \tau^{(i)} , \\
m^{(n+1)}_0 &:= \left\lceil \left\| u \left( T^{(n)} \right) \right\|_{H^1_0(\T)} \right\rceil \vee m^{(n)}_0 , \\
\tau^{(n+1)} &:= \tau_{m^{(n+1)}_0} \left( W^Q\left( T^{(n)} + \cdot \right) - W^Q \left( T^{(n)} \right) \right) ,\\
T^* &:= \sup_{n \geq 1} T^{(n)},
\end{align*}
where at each iteration we use Lemma \ref{localWP} to extend the process $(u(t))_{t\in\left[0,T^{(n)}\right]}$ to the unique mild solution of Equation \eqref{SSCL} on $[0,T^{(n)}]$. It is then clear that $(u(t))_{t \in [0, T^*)}$ satisfies the first assertion of Lemma \ref{lem:solmax}.

Since the sequence of integers $( m^{(n)}_0 )_{n\geq0}$ is nondecreasing, $\sup_{n\geq0} m^{(n)}_0<+\infty$ if and only if there exists $n_0\geq0$ and $m\geq0$ such that, for all $n\geq n_0$, $m^{(n)}_0=m$. Hence, we can write
\begin{align*}
\left\{ T^* < +\infty , \sup_{n\geq0} m^{(n)}_0 < +\infty \right\} &= \bigcup_{n_0\geq0,m\geq0} \left\{ \sum_{n=0}^\infty \tau^{(n)} < +\infty, \forall n \geq n_0 , m^{(n)}_0 = m \right\} \\
&= \bigcup_{n_0\geq0,m\geq0} \left\{ \sum_{n=n_0+1}^\infty \tau^{(n)} < +\infty , \forall n\geq n_0 , m^{(n)}_0 = m \right\} \\
&\subset \bigcup_{n_0\geq0,m\geq0} \left\{ \sum_{n=n_0+1}^\infty \tau_m \left( W^Q \left( T^{(n)} + \cdot \right) - W^Q \left( T^{(n)} \right) \right) < +\infty \right\} .
\end{align*}
However, by the strong Markov property, for any $m\geq0$, the random variables $\tau_m ( W^Q ( T^{(n)} + \cdot ) - W^Q ( T^{(n)} ) )$, $n\geq1$, are independent and identically distributed, and by the definition of $\tau_m(\cdot)$, they are almost surely positive. As a consequence, by Borel's 0-1 law,
\[ \forall n_0 , m \geq 0 , \quad \P \left( \sum_{n=n_0+1}^\infty \tau_m \left( W^Q \left( T^{(n)} + \cdot \right) - W^Q \left( T^{(n)} \right) \right) < +\infty \right) = 0 . \]
As the countable union of negligible events is still negligible, we get
\[ \P \left( T^* < +\infty , \sup_{n\geq0} m^{(n)}_0 < +\infty \right) = 0 . \]
This implies that almost surely, if $T^* < +\infty$ then $\sup_{n\geq0} m^{(n)}_0 = +\infty$, so that $\limsup_{n\to \infty} \|u(T^{(n)})\|_{H^1_0(\T)} = +\infty$, which is the wanted result.
\end{proof}

\subsection{Estimates on the maximal solution}\label{sec:estimates} Let $u_0 \in H^2_0(\T)$. Let $(T^*,(u(t))_{t \in [0,T^*)})$ be the maximal solution to Equation \eqref{SSCL} given by Lemma \ref{lem:solmax}. By Proposition \ref{prop:mild-strong}, $(u(t))_{t \in [0,T^*)}$ is a continuous $H^2_0(\T)$-valued process. Besides, Lemma \ref{lem:solmax} allows us to define, for any $r\geq0$, the stopping time
\begin{equation}\label{stoppingtime}
T_r := \inf \left\{ t \in \left[0,T^*\right) : \left\| u(t) \right\|_{H^1_0(\T)}^2 \geq r \right\} ,
\end{equation}
which always satisfies $T_r\leq T^*$. In the sequel, we shall prove that $\lim_{r\to\infty} T_r = +\infty$, which shall imply that $T^* = +\infty$, almost surely.

\begin{lemma}\label{momentapprox}
Under Assumptions \ref{ass:A} and \ref{ass:g}, for any $p\in 2\N^*$ and for all $t\geq0$, we have:
\begin{equation}\label{intermapprox}
\frac{4\nu}{p} (p-1) \E \left[ \int_0^{t\wedge T_r} \int_\T \left( \partial_x u(s)^{p/2} \right)^2 \dd x \dd s \right]  \leq \| u_0 \|_{L^p_0(\T)}^p + \frac{p(p-1)}{2} D_0 \E \left[ \int_0^{t\wedge T_r} \| u(s) \|_{L^{p-2}_0(\T)}^{p-2} \dd s \right] .
\end{equation}
Moreover, there exist two constants $C_5^{(p)},C_6^{(p)}>0$ depending only on $\nu$, $p$ and $D_0$ such that
\begin{equation}\label{momentfinitetime}
\E \left[ \int_0^{t\wedge T_r} \left\| u(s) \right\|_{L^p_0(\T)}^p \dd s \right] \leq C_5^{(p)} \left( 1+  \| u_0 \|_{L^p_0(\T)}^p \right) + C_6^{(p)} t .
\end{equation}
\end{lemma}
\begin{proof}
Let $p\in 2\N^*$. We want to apply It\^o's formula on $[0,t \wedge T_r]$ to the $H^2_0(\T)$-valued process $(u(t))_{t \in [0,T^*)}$ with the function $F_p : u \mapsto \|u\|^p_{L^p_0(\T)}$. Since this process writes
\[ u(t) = u_0 + \int_0^t \varphi(s)\dd s + W^Q(t) \]
with $\varphi(t) = -\partial_x A(u(t)) + \nu \partial_{xx} u(t) \in L^2_0(\T)$, the standard formulation of It\^o's formula in Hilbert spaces \cite[Theorem 4.32]{DZ92} requires at least $F_p$ to be continuous on $L^2_0(\T)$, which is not the case for $p > 2$ here. Hence, we shall proceed to approximate $F_p$ with a sequence of smooth functions $F_{M,p}$, $M \geq 1$, apply It\^o's formula to the functions $F_{M,p}$ and then take the limit $M \to +\infty$.

\textbf{Step 1. Approximation of the $L^p_0(\T)$-norm.} Let $\rho$ be a $C^\infty$ function from $\R$ to $\R_+$ such that $\int_\R \rho(u)\dd u = 1$ and whose support is contained in the interval $(-\frac12,\frac12)$. For any $M\geq1$, we set the regularised Heaviside function $\psi_M:= \ind{\left(-\infty,M+\frac12\right]}*\rho$ and its antiderivative
\[ \phi_M : u \in \R_+ \longmapsto \int_0^u \psi_M(v)\dd v \in \R_+ . \]
We now define a truncated $L^p_0(\T)$-norm by setting
\begin{equation*}
 F_{M,p} :
 \begin{cases}
  L^2_0(\T) &\longrightarrow \R_+ \\
  v &\longmapsto \int_\T \phi_M \left( v(x)^p \right) \dd x .
 \end{cases}
\end{equation*}
The first differential $\DD F_{M,p}$ and the second differential $\DD^2F_{M,p}$ have the following expressions: $\forall v,h\in L^2_0(\T)$,
\[ \langle \DD F_{M,p}(v),h \rangle_{L^2_0(\T)} =p\int_\T h(x)v(x)^{p-1}\phi'_M\left( v(x)^p \right) \dd x , \]
\[ \langle \DD^2F_{M,p}(v)\cdot h,h\rangle_{L^2_0(\T)}=p(p-1)\int_\T h(x)^2v(x)^{p-2}\phi'_M\left(v(x)^p\right)\dd x+p^2\int_\T h(x)^2v(x)^{2(p-1)}\phi''_M\left(v(x)^p\right)\dd x . \]

\textbf{Step 2. It\^o's formula.} First, let us notice that the process $(W^Q(t))_{t\geq0}$ can be seen as an $L^2_0(\T)$-valued $Q'$-Wiener process where the operator $Q' : L^2_0(\T)\to L^2_0(\T)$ has covariance
\[ \langle u, Q'v\rangle_{L^2_0(\T)} = \sum_{k \geq 1} \langle g_k,u\rangle_{L^2_0(\T)}\langle g_k,v\rangle_{L^2_0(\T)} . \]
Indeed, Assumption \ref{ass:g} ensures that $Q'( L^2_0(\T))\subset H^2_0(\T)$ and $Q'_{|H^2_0(\T)} = Q$. We now have
\[ \mathrm{Tr} \left( \DD^2F_{M,p}(v) Q' \right) = \sum_{k\geq1} \langle \DD^2F_{M,p}(v)g_k,g_k \rangle_{L^2_0(\T)} , \]
so that we can apply It\^o's formula \cite[Theorem 4.32]{DZ92} for the real-valued process $\left(F_{M,p}(u(t))\right)_{t \in [0,T^*)}$, which leads to
\begin{align*}
 F_{M,p}\left( u(t) \right) = &F_{M,p}(u_0) +p\int_0^t \int_\T \left( -\partial_xA(u(s)) +\nu\partial_{xx}u(s) \right) u(s)^{p-1} \phi'_M\left(u(s)^p\right)\dd x\dd s \\
 &+\int_0^t\langle \DD F_{M,p}(u(s)) , \dd W^Q(s) \rangle_{L^2_0(\T)} \\
 &+\frac12 p(p-1)\sum_{k\geq1} \int_0^t \int_\T g_k^2 u(s)^{p-2} \phi'_M(u(s)^p) \dd x\dd s\\
 &+\frac12 p^2\sum_{k\geq1}\int_0^t\int_\T g_k^2 u(s)^{2(p-1)} \phi''_M (u(s)^p) \dd x\dd s .
\end{align*}
Since the $L^2_0(\T)$-norm of $\DD F_{M,p}(u(s))$ is bounded uniformly in time, the third term of the right-hand side is a square integrable martingale \cite[Theorem 4.27]{DZ92}. Thus, for $t \geq 0$, integrating in time up to $t\wedge T_r$ and taking the expectation, we get
\begin{align}
 \E \left[ F_{M,p}\left( u(t\wedge T_r) \right) \right] = &F_{M,p}(u_0) -p \E \left[ \int_0^{t\wedge T_r} \int_\T \partial_xA(u(s)) u(s)^{p-1} \phi'_M\left(u(s)^p\right)\dd x\dd s \right] \label{itom:driftflux} \\
 &+p \E \left[ \int_0^{t\wedge T_r} \int_\T  \nu\partial_{xx}u(s) u(s)^{p-1} \phi'_M\left(u(s)^p\right)\dd x\dd s \right] \label{itom:driftvisc} \\
 &+\frac12 p(p-1) \E \left[ \sum_{k\geq1} \int_0^{t\wedge T_r} \int_\T g_k^2 u(s)^{p-2} \phi'_M(u(s)^p) \dd x\dd s \right] \label{itom:ito1} \\
 &+\frac12 p^2\E\left[ \sum_{k\geq1}\int_0^{t\wedge T_r}\int_\T g_k^2 u(s)^{2(p-1)} \phi''_M (u(s)^p) \dd x\dd s \right] . \label{itom:ito2}
\end{align}

\textbf{Step 3. Passing $M\to+\infty$.} We want now to pass to the limit $M\to+\infty$. Regarding the left-hand side in the above equation, the family of functions $\phi_M$ is non-decreasing with respect to $M$, so that the monotone convergence theorem yields
\[ \lim_{M\to\infty} \E \left[ F_{M,p}\left( u(t\wedge T_r) \right) \right] = \E\left[ \int_\T \lim_{M\to\infty} \phi_M\left( u(t \wedge T_r)^p \right)\dd x \right] =\E\left[ \left\| u(t\wedge T_r) \right\|_{L^p_0(\T)}^p \right]. \]
For the flux term, we have almost surely, for all $s\in[0,t\wedge T_r]$ and for all $M\geq0$, $\partial_xA(u(s))u(s)^{p-1}\phi'_M(u(s)^p) \leq |\partial_xA(u(s))||u(s)|^{p-1}$. Furthermore,
\begin{align*}
\E \left[ \int_0^{t\wedge T_r} \int_\T \left| \partial_xA(u(s)) \right| |u(s)|^{p-1}\dd x \dd s \right] &\leq \E \left[ \sup_{s\in[0,t\wedge T_r]} \|u(s)\|_{L^\infty_0(\T)}^{p-1} \int_0^{t\wedge T_r} \int_\T \left| \partial_xA(u(s)) \right| \dd x \dd s \right] \\
&\leq r^{\frac{p-1}2} \E \left[ \int_0^{t\wedge T_r} \left\| \partial_xA(u(s)) \right\|_{L^2_0(\T)} \dd s \right] \quad \text{(from \eqref{strongerpoincare} and \eqref{stoppingtime})} \\
&\leq r^{\frac{p-1}2} L_r \E \left[ \int_0^{t\wedge T_r} \left\| u(s) \right\|_{L^2_0(\T)} \dd s \right] \quad \text{(from Proposition \ref{localboundedness})} \\
&\leq L_r r^{\frac p2} t < +\infty .
\end{align*}
Thus, the dominated convergence theorem applies and yields
\[ \lim_{M\to\infty} p \E \left[ \int_0^{t\wedge T_r} \int_\T \partial_xA(u(s)) u(s)^{p-1} \phi'_M\left(u(s)^p\right)\dd x\dd s \right] = p \E \left[ \int_0^{t\wedge T_r} \int_\T \partial_xA(u(s)) u(s)^{p-1} \dd x\dd s \right] . \]
We now integrate by parts the viscous term:
\begin{align*}
&p\nu \E \left[ \int_0^{t\wedge T_r} \int_\T \partial_{xx}u(s) u(s)^{p-1} \phi'_M\left( u(s)^p\right) \dd x \dd s \right] \\
&\quad = -p\nu \E \left[ \int_0^{t\wedge T_r} \int_\T \partial_xu(s) \left( \partial_x\left( u(s)^{p-1} \right) \phi'_M\left(u(s)^p\right) + u(s)^{p-1} \partial_x \left( \phi'_M\left(u(s)^p\right)\right) \right) \dd x \dd s \right] \\
&\quad = -p\nu \E \left[ \int_0^{t\wedge T_r} \int_\T (\partial_xu(s))^2 \left( (p-1)u(s)^{p-2} \phi'_M\left( u(s)^p\right) + p u(s)^{2(p-1)} \phi''_M\left(u(s)^p\right) \right) \dd x \dd s \right] , \\
\end{align*}
and this last integrand is dominated uniformly in $M$ by $(\partial_xu(s))^2 \left( (p-1)u(s)^{p-2} + \kappa p u(s)^{2(p-1)} \right)$, where $\kappa = \sup_\R |\rho|$. Furthermore, thanks to \eqref{stoppingtime}, we have
\begin{align*}
&\E \left[ \int_0^{t\wedge T_r} \int_\T \left( \partial_xu(s)\right)^2 \left( (p-1)u(s)^{p-2}+\kappa pu(s)^{2(p-1)} \right)\dd x \dd s \right] \\
&\quad \leq \E \left[ \left( (p-1) \sup_{s\in[0,t\wedge T_r]} \| u(s)\|_{L^\infty_0(\T)}^{p-2}+\kappa p\sup_{s\in[0,t\wedge T_r]} \| u(s)\|_{L^\infty_0(\T)}^{2(p-1)}\right) \int_0^{t\wedge T_r} \|u(s)\|_{H^1_0(\T)}^2 \dd s \right] \\
&\quad \leq \left( (p-1) r^{\frac{p-2}2} + \kappa p r^{p-1} \right) rt <+\infty.
\end{align*}
Thus, we get from the dominated convergence theorem,
\[ \lim_{M\to\infty} p \E \left[ \int_0^{t\wedge T_r} \int_\T  \nu\partial_{xx}u(s) u(s)^{p-1} \phi'_M\left(u(s)^p\right)\dd x\dd s \right] = -\nu p(p-1) \E \left[ \int_0^{t\wedge T_r} \int_\T (\partial_xu(s))^2 u(s)^{p-2} \dd x \dd s\right] . \]

With similar computations, for the noise term, we have
\[ \lim_{M\to\infty} p(p-1) \E \left[ \sum_{k\geq1} \int_0^{t\wedge T_r} \int_\T g_k^2 u(s)^{p-2} \phi'_M(u(s)^p) \dd x\dd s \right] = p(p-1)\E\left[ \sum_{k\geq1} \int_0^{t\wedge T_r} \int_\T g_k^2 u(s)^{p-2} \dd x\dd s \right] , \]
and
\[ \lim_{M\to\infty} p^2\E\left[ \sum_{k\geq1}\int_0^{t\wedge T_r} \int_\T g_k^2 u(s)^{2(p-1)} \phi''_M (u(s)^p) \dd x\dd s\right] = 0 . \]
Letting $M$ go to $+\infty$ in \eqref{itom:driftflux}, \eqref{itom:driftvisc}, \eqref{itom:ito1} and \eqref{itom:ito2}, we get
\begin{multline}\label{dynkinapprox}
\E \left[ \| u(t\wedge T_r) \|_{L^p_0(\T)}^p \right] = \| u_0 \|_{L^p_0(\T)}^p -p \E\left[ \int_0^{t\wedge T_r} \int_\T \partial_x A(u(s)) u(s)^{p-1}\dd x \dd s \right] \\
-\nu p(p-1) \E \left[ \int_0^{t\wedge T_r} \int_\T (\partial_xu(s))^2 u(s)^{p-2}\dd x \dd s \right] +\frac12p(p-1)\sum_{k\geq1}\E\left[ \int_0^{t\wedge T_r} \int_\T u(s)^{p-2}g_k^2\dd x \dd s \right] .
\end{multline}

It turns out that the flux term disappears:
\begin{equation}\label{fluxdisappears}
\int_\T u(s)^{p-1} \partial_x A(u(s)) \dd x = \int_\T u(s)^{p-1} A'(u(s)) \partial_x u(s) \dd x = \int_\T \partial_x \left( \mathcal A_p (u(s)) \right) \dd x = 0 ,
\end{equation}
where $\mathcal A_p$ is an antiderivative of $v \mapsto v^{p-1} A'(v)$. As regards the noise coefficients, we have
\[ \sum_{k\geq1} g_k(x)^2 \leq \sum_{k\geq1} \| g_k \|_{L^\infty_0(\T)}^2 \leq \sum_{k\geq1} \| g_k \|_{H^1_0(\T)}^2 \leq D_0 , \]
thanks to \eqref{strongerpoincare} and \eqref{boundgk}. As a consequence, we get from \eqref{dynkinapprox} the inequality
\begin{equation}\label{someineq}
\nu p (p-1) \E \left[ \int_0^{t\wedge T_r} \int_\T \left( u(s)^{\frac p2-1} \partial_x u(s) \right)^2 \dd x \dd s \right] \leq \| u_0 \|_{L^p_0(\T)}^p + \frac 1 2 p (p-1) D_0 \E \left[ \int_0^{t\wedge T_r} \| u(s) \|_{L^{p-2}_0(\T)}^{p-2} \dd s \right] .
\end{equation}
Rewriting the integrand in the left-hand side, we get
\begin{equation}
\frac{4\nu}{p} (p-1) \E \left[ \int_0^{t\wedge T_r} \int_\T \left( \partial_x\left( u(s)^{p/2}\right) \right)^2 \dd x \dd s \right]  \leq \| u_0 \|_{L^p_0(\T)}^p + \frac{p(p-1)}{2} D_0 \E \left[ \int_0^{t\wedge T_r} \| u(s) \|_{L^{p-2}_0(\T)}^{p-2} \dd s \right] .
\end{equation}
Since $u(s)$ has a zero space average   and is continuous in space (because it belongs to $H^1_0(\T)$), almost surely the function $u(s)^{p/2}$ vanishes somewhere on the torus. Thus, we can apply the Poincar\'e inequality on the left-hand side which leads, after multiplying by $p/(4\nu(p-1))$ on both sides, to the inequality
\begin{equation}\label{Lqbound}
\E \left[ \int_0^{t\wedge T_r} \| u(s) \|_{L^p_0(\T)}^p \dd s \right] \leq \frac{p}{4\nu(p-1)} \| u_0 \|_{L^p_0(\T)}^p + \frac{p^2D_0}{8\nu} \E \left[ \int_0^{t\wedge T_r} \| u(s) \|_{L^{p-2}_0(\T)}^{p-2} \dd s \right] .
\end{equation}
For $p=2$, we get
\[ \E \left[ \int_0^{t\wedge T_r} \| u(s) \|_{L^2_0(\T)}^2 \dd s \right] \leq \frac1{2\nu} \| u_0 \|_{L^2_0(\T)}^2 + \frac{D_0t}{2\nu} , \]
and the claimed result for arbitrary $p\in 2\N^*$ follows by induction and from the inequalities $\| u_0 \|_{L^{p-2r}_0(\T)}^{p-2r} \leq 1 + \| u_0 \|_{L^p_0(\T)}^p$ and $\E [ t \wedge T_r ] \leq t$.
\end{proof}

\begin{remark}
By Jensen's inequality, the bound \eqref{momentfinitetime} also holds for any real number $p\geq2$.
\end{remark}

\begin{lemma}\label{H1approx}
Under Assumptions \ref{ass:A} and \ref{ass:g}, there exist two constants $C_7,C_8>0$ depending only on $\nu$, $p_A$, $C_1$ and $D_0$, such that for all $t\geq0$ and all $r\geq0$,
\[ \E \left[ \| u (t\wedge T_r) \|_{H^1_0(\T)}^2 \right] + \nu \E \left[ \int_0^{t\wedge T_r} \| u(s) \|_{H^2_0(\T)}^2 \dd s\right] \leq \| u_0 \|_{H^1_0(\T)}^2 + C_7 \left( 1+ \| u_0 \|_{L^{2p_A+2}_0(\T)}^{2p_A+2} \right) + C_8 t . \]
\end{lemma}
\begin{proof}
We want to apply It\^o's formula to the squared $H^1_0(\T)$-norm of the process $\left(u(t)\right)_{t\in[0,T^*)}$. As for the proof of Lemma \ref{momentapprox}, we proceed by truncation of this function.

\textbf{Step 1. Approximation of the $H^1_0(\T)$-norm.} We set
\begin{equation*}
G_M :
 \begin{cases}
  L^2_0(\T) &\longrightarrow \R_+ \\
  v &\longmapsto \sum_{m=1}^M (-\lambda_m) \langle v, e_m \rangle_{L^2_0(\T)}^2
 \end{cases}
\end{equation*}
The first differential $\DD G_M$ and the second differential $\DD^2G_M$ have the following expressions: $\forall h \in L^2_0(\T)$,
\[ \langle \DD G_M(v),h \rangle_{L^2_0(\T)} = -2\sum_{m=1}^M \lambda_m \langle v, e_m \rangle_{L^2_0(\T)} \langle h,e_m \rangle_{L^2_0(\T)} , \]
\[ \langle \DD^2 G_M (v) \cdot h,h \rangle_{L^2_0(\T)} = -2 \sum_{m=1}^M \lambda_m \langle h,e_m \rangle_{L^2_0(\T)}^2 . \]

\textbf{Step 2. It\^o's formula.} It\^o's formula applied to $G_M$ yields almost surely and for all $r\geq0$,
\begin{multline}\label{itoH1}
 G_M(u(t\wedge T_r)) = G_M(u_0) -2\int_0^{t\wedge T_r} \sum_{m=1}^M \lambda_m \langle u(s),e_m\rangle_{L^2_0(\T)} \langle -\partial_x A(u(s))+\nu\partial_{xx}u(s) , e_m\rangle_{L^2_0(\T)} \dd s \\
 -2\int_0^{t\wedge T_r}\langle \DD G_M(u(s)) , \dd W^Q(s) \rangle_{L^2_0(\T)} - 2 \sum_{k\geq1} \int_0^{t\wedge T_r} \sum_{m=1}^M \lambda_m \langle g_k ,e_m \rangle_{L^2_0(\T)}^2 \dd s .
\end{multline}
We first check that the third term of the right-hand side is a square-integrable martingale:
\begin{align*}
\E \left[ \int_0^{t\wedge T_r} \left\| DG_M (u(s)) \right\|_{L^2_0(\T)}^2 \dd s \right] &= 4 \sum_{m=1}^M \lambda_m^2  \E \left[ \int_0^{t\wedge T_r} \langle u(s),e_m \rangle_{L^2_0(\T)}^2 \dd s \right] \\
&\leq 4 \left( \sum_{m=1}^M \lambda_m^2 \right) \E \left[ \int_0^{t\wedge T_r} \|u(s)\|_{L^2_0(\T)}^2 \dd s \right] \leq 4 \left( \sum_{m=1}^M \lambda_m^2 \right) t r < +\infty .
\end{align*}
Thus, taking the expectation, the stochastic integral disappears and we get
\begin{multline}\label{itoH1E}
\E \left[ G_M(u(t \wedge T_r)) \right] = G_M(u_0) +2 \E \left[ \int_0^{t\wedge T_r} \sum_{m=1}^M \lambda_m \langle u(s),e_m\rangle_{L^2_0(\T)} \langle \partial_x A(u(s)) , e_m\rangle_{L^2_0(\T)} \dd s \right] \\
-2 \E \left[ \int_0^{t\wedge T_r} \sum_{m=1}^M \lambda_m \langle u(s),e_m\rangle_{L^2_0(\T)} \langle \nu\partial_{xx}u(s) , e_m\rangle_{L^2_0(\T)} \dd s \right] - \E \left[ \sum_{k\geq1} \int_0^{t\wedge T_r} \sum_{m=1}^M \lambda_m \langle g_k ,e_m \rangle_{L^2_0(\T)}^2 \dd s \right] .
\end{multline}

On one hand, we can rewrite the viscous term as follows:
\begin{align}
  \sum_{m=1}^M \lambda_m \langle u(s),e_m\rangle_{L^2_0(\T)}\langle \nu \partial_{xx} u(s),e_m\rangle_{L^2_0(\T)} &= \sum_{m=1}^M \lambda_m \langle u(s),e_m\rangle_{L^2_0(\T)}\langle \nu u(s),\partial_{xx}e_m\rangle_{L^2_0(\T)}\nonumber \\
  &= \sum_{m=1}^M \lambda_m \langle u(s),e_m\rangle_{L^2_0(\T)}\langle \nu u(s),\lambda_me_m\rangle_{L^2_0(\T)}\nonumber \\
  &= \nu \sum_{m=1}^M \lambda_m^2 \langle u(s),e_m\rangle_{L^2_0(\T)}^2 . \label{IBPvisc}
\end{align}
On the other hand, applying Young's inequality on the flux term, we get
\begin{multline}\label{IBPflux}
2 \E \left[ \int_0^{t\wedge T_r} \sum_{m=1}^M \lambda_m \langle u(s),e_m\rangle_{L^2_0(\T)} \langle \partial_x A(u(s)) , e_m\rangle_{L^2_0(\T)} \dd s \right] \\
\leq 2\nu \E \left[ \int_0^{t\wedge T_r} \sum_{m=1}^M \lambda_m^2 \langle u(s),e_m\rangle_{L^2_0(\T)}^2 \dd s \right] + \frac1{2\nu} \E \left[ \int_0^{t\wedge T_r} \sum_{m=1}^M \langle \partial_x A(u(s)) , e_m\rangle_{L^2_0(\T)}^2 \dd s \right] .
\end{multline}
Injecting \eqref{IBPvisc} and \eqref{IBPflux} into \eqref{itoH1E}, we get the inequality
\begin{equation}\label{BL}
\E \left[ G_M(u(t\wedge T_r)) \right] \leq G_M(u_0) + \frac1{2\nu} \E \left[ \int_0^{t\wedge T_r} \sum_{m=1}^M \langle \partial_x A(u(s)) , e_m\rangle_{L^2_0(\T)}^2 \dd s \right] \\
 - \E[t\wedge T_r]\sum_{k\geq1} \sum_{m=1}^M \lambda_m \langle g_k ,e_m \rangle_{L^2_0(\T)}^2 .
\end{equation}

\textbf{Step 3. Passing $M\to+\infty$.} From Proposition \ref{localboundedness}, for any $r\geq0$, there is a constant $L_r$ such that for all $M\geq1$, we have
\[ \sum_{m=1}^M \langle \partial_x A(u(s)) ,e_m\rangle_{L^2_0(\T)}^2 \leq \|\partial_xA(u(s))\|_{L^2_0(\T)}^2 \leq L_r \|u(s)\|_{H^1_0(\T)}^2 \leq rL_r . \]
Thus, we can use the dominated convergence theorem to let $M$ go to infinity in \eqref{BL} and we get
\begin{equation}\label{dynkinH1}
\E \left[ \| u(t\wedge T_r) \|_{H^1_0(\T)}^2 \right] \leq \| u_0 \|_{H^1_0(\T)}^2 +\frac1{2\nu} \E \left[ \int_0^{t\wedge T_r} \|\partial_xA(u(s)) \|_{L^2_0(\T)}^2 \dd s \right] + \E[t\wedge T_r] \sum_{k\geq1}\|g_k\|_{H^1_0(\T)}^2 .
\end{equation}

Since from Assumption \ref{ass:A}, $A'$ has polynomial growth, we can bound the second term of the right-hand side: using \eqref{PGA} and \eqref{intermapprox} with $p=2$ and $p=2p_A+2$, we get
\begin{align*}
\E \left[ \int_0^{t\wedge T_r} \| \partial_xA(u(s)) \|_{L^2_0(\T)}^2 \dd s \right] &= \E \left[ \int_0^{t \wedge T_r} \int_\T (\partial_x u(s))^2 A'(u(s))^2 \dd x \dd s \right] \\
&\leq 2C_1^2 \E \left[ \int_0^{t \wedge T_r} \int_\T (\partial_x u(s))^2 \left( 1+|u(s)|^{2p_A} \right) \dd x \dd s \right] \\
&= 2C_1^2 \left( \E \left[ \int_0^{t \wedge T_r} \| u(s) \|_{H^1_0(\T)}^2 \dd s \right] + \E \left[ \int_0^{t \wedge T_r} \int_\T (\partial_x u(s))^2 u(s)^{2p_A} \dd x\dd s \right] \right) \\
&\leq \frac{C_1^2}\nu \Big(\|u_0\|_{L^2_0(\T)}^2 + D_0 \E [t\wedge T_r] \\
&\quad + \frac{2}{(2p_A+2)(2p_A+1)} \| u_0 \|_{L^{2p_A+2}_0(\T)}^{2p_A+2} + D_0 \E \left[ \int_0^{t \wedge T_r} \|u(s)\|_{L^{2p_A}_0(\T)}^{2p_A} \dd s \right] \Big) .
\end{align*}
Applying now Lemma \ref{momentapprox}, we get
\begin{multline*}
\E \left[ \int_0^{t\wedge T_r} \| \partial_xA(u(s)) \|_{L^2_0(\T)}^2 \dd s \right] \leq \frac{C_1^2}\nu \Big( 2 \left( 1+ \| u_0 \|_{L^{2p_A+2}_0(\T)}^{2p_A+2} \right) + D_0 t \\
+D_0 C_5^{(2p_A)} \left( 1+ \|u_0\|_{L^{2p_A}_0(\T)}^{2p_A} \right) + C_6^{(2p_A)} t \Big) .
\end{multline*}
Injecting this last bound in \eqref{dynkinH1}, we get the wanted result.
\end{proof}

\begin{corollary}[Limit of $T_r$]\label{cor:Tr}
  Under Assumptions \ref{ass:A} and \ref{ass:g}, $T_r \to +\infty$ almost surely, and thus $T^*=+\infty$ almost surely.
\end{corollary}
\begin{proof}
Let $t\geq0$. Writing
\[ \P\left(T_r<t\right)=\P\left(\left\|u(t\wedge T_r)\right\|_{H^1_0(\T)}^2 \geq r\right) , \]
we get from Markov's inequality,
\[ \P\left(T_r<t\right)\leq\frac1r\E\left[\left\|u(t\wedge T_r)\right\|_{H^1_0(\T)}^2\right] . \]
We apply now Lemma \ref{H1approx} to get
\[ \P \left( T_r<t \right) \leq \frac1r \left( \|u_0\|_{H^1_0(\T)}^2 +C_7\left(1+\|u_0\|_{L^{2p_A+2}_0(\T)}^{2p_A+2} \right) +C_8t \right) \underset{r\to\infty}\longrightarrow0 . \]
Since $t$ has been chosen arbitrarily, it follows that almost surely, $T_r$ tends to $+\infty$ as $r\to+\infty$. Then, since $T_r \leq T^*$, we have $T^*=+\infty$ almost surely.
\end{proof}

\subsection{Proof of Theorem \ref{wellposedness}}\label{proofthm1} Under Assumptions \ref{ass:A} and \ref{ass:g}, let $u_0 \in H^2_0(\T)$, and $(T^*,(u(t))_{t \in [0,T^*)})$ be the maximal solution to Equation \eqref{SSCL} given by Lemma \ref{lem:solmax}. By Corollary \ref{cor:Tr}, $T^*=+\infty$ almost surely. Therefore, $(u(t))_{t \geq 0}$ is the unique (global) mild solution to Equation \eqref{SSCL}, and by Proposition \ref{prop:mild-strong}, it is also the unique (global) strong solution to this equation. It remains to check that this solution depends continuously on $u_0$.

\begin{lemma}[Continuous dependence on initial conditions]\label{cd}
If $( u_0^{(j)} )_{j\geq1}$ is a sequence of $H^2_0(\T)$ satisfying
\[ \lim_{j\to\infty} \left\| u_0 - u^{(j)}_0 \right\|_{H^2_0(\T)} = 0 , \]
then, denoting by $( u^{(j)}(t) )_{t\geq0,j\geq1}$ the family of associated solutions, for any $T\geq0$, we have almost surely
\[ \lim_{j\to\infty} \sup_{t\in[0,T]} \left\| u(t) - u^{(j)}(t) \right\|_{H^2_0(\T)} = 0 . \]
\end{lemma}
\begin{proof}
Let us fix a time horizon $T>0$. Subtracting the mild formulations of $\left(u(t)\right)_{t\geq0}$ and $(u^{(j)}(t))_{t\geq0}$ given by Proposition \ref{prop:mild-strong} and taking the $H^2_0(\T)$-norm, we get by the triangle inequality and Proposition \ref{prop:heatkernel}, for all $t\in[0,T]$,
\begin{align}
\left\| u(t)-u^{(j)}(t) \right\|_{H^2_0(\T)} &\leq \left\| S_t\left(u_0-u^{(j)}_0\right) \right\|_{H^2_0(\T)} + \int_0^t \left\| S_{t-s} \partial_x \left( A(u(s))-A\left(u^{(j)}(s)\right) \right) \right\|_{H^2_0(\T)}  \nonumber \\
&\leq \left\| u_0-u^{(j)}_0 \right\|_{H^2_0(\T)} + \int_0^t \frac{C_4}{\sqrt{t-s}} \left\| \partial_xA(u(s))-\partial_xA\left(u^{(j)}(s)\right) \right\|_{H^1_0(\T)} \dd s . \label{mildcoupling}
\end{align}
Now, for any $M>0$, we define the stopping times
\[ \tau_M := \inf \left\{ t\geq0: \| u(t) \|_{H^2_0(\T)}\geq M \right\} , \quad \tau^{(j)}_M := \inf \left\{ t\geq0: \| u^{(j)}(t) \|_{H^2_0(\T)} \geq M \right\} ,\quad j\in \N , \]
and we denote by $L_M$, according to Proposition \ref{localboundedness}, the Lipschitz constant of the mapping $v \in H^2_0(\T) \mapsto \partial_x A(v) \in H^1_0(\T)$ over the centered ball in $H^2_0(\T)$ of radius $M$. For an arbitrarily fixed $t\in[0,T]$, the inequality \eqref{mildcoupling} implies
\begin{multline*}
\left\| u\left(t\wedge\tau_M\wedge\tau^{(j)}_M\right)-u^{(j)}\left(t\wedge\tau_M\wedge\tau^{(j)}_M\right) \right\|_{H^2_0(\T)} \leq \left\| u_0-u^{(j)}_0 \right\|_{H^2_0(\T)}\\
+ \int_0^{t\wedge\tau_M\wedge\tau^{(j)}_M} \frac{C_4L_M}{\sqrt{t\wedge\tau_M\wedge\tau^{(j)}_M-s}} \left\| u(s)-u^{(j)}(s) \right\|_{H^2_0(\T)} \dd s .
\end{multline*}
In the next step, we iterate this last inequality and apply the Fubini theorem on the double time integral:
\begin{align*}
&\left\| u\left(t\wedge\tau_M\wedge\tau^{(j)}_M\right)-u^{(j)}\left(t\wedge\tau_M\wedge\tau^{(j)}_M\right) \right\|_{H^2_0(\T)} \\
&\leq \left\| u_0-u^{(j)}_0 \right\|_{H^2_0(\T)} \left( 1+2\sqrt{t\wedge\tau_M\wedge\tau^{(j)}_M}C_4L_M \right) \\
&\quad+ C_4^2L_M^2 \int_0^{t\wedge\tau_M\wedge\tau^{(j)}_M}\int_0^s \frac1{\sqrt{(t\wedge\tau_M\wedge\tau^{(j)}_M-s)(s-r)}} \left\| u(r)-u^{(j)}(r) \right\|_{H^2_0(\T)} \dd r \dd s \\
&\leq \left\| u_0-u^{(j)}_0 \right\|_{H^2_0(\T)} \left( 1+2\sqrt T C_4L_M \right) \\
&\quad+ C_4^2L_M^2 \int_0^{t\wedge\tau_M\wedge\tau^{(j)}_M}\left( \int_r^{t\wedge\tau_M\wedge\tau^{(j)}_M} \frac1{\sqrt{(t\wedge\tau_M\wedge\tau^{(j)}_M-s)(s-r)}} \dd s \right) \left\| u(r)-u^{(j)}(r) \right\|_{H^2_0(\T)} \dd r .
\end{align*}
However, by a change of variable, we have
\[ \int_s^{t\wedge\tau_M\wedge\tau^{(j)}_M} \frac1{\sqrt{(t\wedge\tau_M\wedge\tau^{(j)}_M-r)(r-s)}} \dd r = \int_{-1}^1 \frac1{\sqrt{1-y^2}} \dd y = \pi . \]
Hence, Gr\"onwall's lemma yields the following control
\[ \left\| u\left(t\wedge\tau_M\wedge\tau^{(j)}_M\right)-u^{(j)}\left(t\wedge\tau_M\wedge\tau^{(j)}_M\right) \right\|_{H^2_0(\T)} \leq \left\| u_0-u^{(j)}_0 \right\|_{H^2_0(\T)} \left( 1+2\sqrt TC_4L_M \right) \e^{C_4^2L_M^2\pi t\wedge\tau_M\wedge\tau^{(j)}_M} . \]
It follows from this inequality that $\liminf_{j\to\infty} \tau^{(j)}_M \geq \tau_M\wedge T$. Indeed, assuming the opposite, we would have (along a subsequence)
\[ \left\| u \left( \tau^{(j)}_M \right) -u^{(j)}\left(\tau^{(j)}_M\right) \right\|_{H^2_0(\T)} \leq \left\| u_0-u^{(j)}_0 \right\|_{H^2_0(\T)} \left(1+2\sqrt TC_4L_M\right)\e^{C_4^2L_M^2T} \underset{j\to\infty}\longrightarrow 0 , \]
which would imply
\[ M \leq \lim_{j\to\infty} \left\| u^{(j)} \left( \tau^{(j)}_M \right) \right\|_{H^2_0(\T)} = \lim_{j\to\infty} \left\| u \left( \tau^{(j)}_M \right) \right\|_{H^2_0(\T)} < M . \]
Hence, necessarily, beyond a certain rank $j$, we have
\[ \left\| u(t\wedge\tau_M) -u^{(j)}(t\wedge\tau_M) \right\|_{H^2_0(\T)} \leq \left\| u_0-u^{(j)}_0 \right\|_{H^2_0(\T)} \left( 1+2\sqrt TC_4L_M \right) \e^{C_4^2L_M^2t\wedge\tau_M} . \]
Since the solutions of \eqref{strongsolution} do not explode, the stopping time $\tau_M$ tends almost surely to $+\infty$ as $M$ tends to $+\infty$. As a consequence, there exists $M_T>0$ such that $T<\tau_{M_T}$ almost surely, so that for all $t\in[0,T]$,
\[ \left\| u(t) -u^{(j)}(t) \right\|_{H^2_0(\T)} \leq \left\| u_0-u^{(j)}_0 \right\|_{H^2_0(\T)} \left( 1+2\sqrt TC_4L_{M_T} \right) \e^{C_4^2L_{M_T}^2T} . \]
Hence the result.
\end{proof}

\section{Invariant measure}\label{IM}

This section is dedicated to the proof of Theorem \ref{thm:im}. The existence of an invariant measure is proven in Subsection \ref{existenceIM} using the Krylov-Bogoliubov theorem, whereas the uniqueness is addressed through a coupling argument relying on the $L^1_0(\T)$-contraction property established in Proposition \ref{L1C}.

The proof of existence of an invariant measure we provide in the next subsection relies plainly on the presence of viscosity. Indeed, the viscous term provides the process $u(t)$ with a dissipative -- and thus a more stable -- behaviour. Still, it has to be borne in mind that when the flux term is nonlinear enough, the presence of a viscous term is not a necessary condition for the stability of the underlying stochastic process. On the physical side, in his theory of turbulent flows \cite{Kol41a,Kol41b}, Kolmogorov already predicted this idea: the statistical distribution of scales of intermediate size in turbulence are not determined by the viscosity coefficient. On the theoretical side, the same idea was validated theoretically by powerful results on the invariant measure for the inviscid stochastic Burgers' equation \cite{EKMS00} and, quite a few years later, for inviscid stochastic conservation laws with "non-degenerate" flux \cite{DV15}. However, our framework differs substantially from the inviscid case in the sense that our stability results are driven by regularity issues which cannot be tackled without viscosity.

\subsection{Preliminary results} By Definition \ref{defi:IM}, an invariant measure for Equation \eqref{SSCL} is a Borel probability measure on $H^2_0(\T)$. Our proofs of existence and uniqueness however involve estimates in various spaces, namely $L^1_0(\T)$, $L^2_0(\T)$ and $H^1_0(\T)$. In particular, we shall manipulate and identify Borel probability measures on these spaces. We first clarify the relation between the associated Borel $\sigma$-fields thanks to the following result. For any metric space $E$, we respectively denote by $\mathcal{B}(E)$ and $\mathcal{P}(E)$ the Borel $\sigma$-field and the set of Borel probability measures on $E$.

\begin{lemma}[Borel probability measures on $L^q_0(\T)$ and $H^s_0(\T)$]\label{lem:sigmafields}
  For all $q \in [1,2]$ and $s \geq 1$, $\mathcal{B}(H^s_0(\T)) = \{B \cap H^s_0(\T) : B \in \mathcal{B}(L^q_0(\T))\}$. As a consequence:
  \begin{enumerate}
    \item[(1)] for any $\mu \in \mathcal{P}(H^s_0(\T))$, the mapping $\tilde\mu(\cdot) = \mu(\cdot\cap H^s_0(\T))$ defines a Borel probability measure on $L^q_0(\T)$;
    \item[(2)] conversely, for any $\tilde\mu \in \mathcal{P}(L^q_0(\T))$ which gives full weight to $H^s_0(\T)$, there exists a unique $\mu \in \mathcal{P}(H^s_0(\T))$ such that $\tilde\mu(B) = \mu(B \cap H^s_0(\T))$ for any $B \in \mathcal{B}(L^q_0(\T))$.
  \end{enumerate}
\end{lemma}
\begin{proof}
  Let $q \in [1,2]$ and $s \geq 1$. The set $\mathcal T$ defined by
  \[ \mathcal{T} = \left\{ B \cap H^s_0(\T) : B \in \mathcal{B}\left( L^q_0(\T) \right)\right\} . \]
  is a $\sigma$-field on $H^s_0(\T)$, called the \textit{trace $\sigma$-field} of $H^s_0(\T)$ in $\mathcal{B}(L^q_0(\T))$.
  
  (1) We denote by $I$ the injection $H^s_0(\T) \to L^q_0(\T)$, so that $\mathcal{T} = \{I^{-1}(B) : B \in \mathcal{B}(L^q_0(\T))\}$. Since $I$ is continuous, and therefore Borel measurable, we have $\mathcal{T} \subset \mathcal{B}(H^s_0(\T))$. Thus, for any $\mu \in \mathcal{P}(H^s_0(\T))$, the pushforward measure $\tilde\mu$ defined by
  \[ \tilde\mu(B) := \mu \circ I^{-1}(B) = \mu\left(B\cap H^s_0(\T)\right) , \qquad B \in \mathcal B \left( L^q_0(\T) \right) , \]
  is a Borel probability measure on $L^q_0(\T)$.
  
  (2) Let us first notice that since $H^s_0(\T)$ is separable, the Borel $\sigma$-field $\mathcal{B}(H^s_0(\T))$ is the smallest $\sigma$-field on $H^s_0(\T)$ containing all closed balls. Let $A \subset H^s_0(\T)$ be such a ball. Since the $H^s_0(\T)$-norm is lower semi-continuous on $L^q_0(\T)$, then $A$ is closed in $L^q_0(\T)$ as a level set of a lower semi-continuous function, and thus $A \in \mathcal B (L^q_0(\T))$. It is then clear that $A \in \mathcal{T}$, which by the minimality property of $\mathcal{B}(H^s_0(\T))$ entails $\mathcal{B}(H^s_0(\T)) \subset \mathcal{T}$, and thus $\mathcal{B}(H^s_0(\T)) = \mathcal{T}$. 
  
  Now let $\tilde\mu$ be a Borel probability measure on $L^q_0(\T)$ which gives full weight to $H^s_0(\T)$, that is to say such that there exists $\tilde{B} \in \mathcal{B}(L^q_0(\T))$ such that $\tilde{B} \subset H^s_0(\T)$ and $\tilde\mu(\tilde{B})=1$. Let us define the Borel probability measure $\mu$ on $H^s_0(\T)$ by
  \[ \mu(B \cap H^s_0(\T)) := \tilde\mu(B) , \qquad B \in \mathcal B \left( L^q_0(\T) \right) . \]
  Notice that this definition is not ambiguous, because the identity $\mathcal{T} = \mathcal{B}(H^s_0(\T))$ ensures that any element of $\mathcal{B}(H^s_0(\T))$ writes under the form $B \cap H^s_0(\T)$ for some $B \in \mathcal{B}(L^q_0(\T))$; besides, if $B_1, B_2 \in \mathcal{B}(L^q_0(\T))$ are such that $B_1 \cap H^s_0(\T) = B_2 \cap H^s_0(\T)$, then $\tilde{\mu}(B_1) = \tilde{\mu}(B_1 \cap \tilde{B}) = \tilde{\mu}(B_2 \cap \tilde{B}) = \tilde{\mu}(B_2)$ because the identity $B_1 \cap H^s_0(\T) = B_2 \cap H^s_0(\T)$ implies that $B_1 \cap \tilde{B} = B_2 \cap \tilde{B}$. Finally, the fact that any $\nu \in \mathcal P(H^s_0(\T))$ such that $\tilde\mu(B) = \nu(B \cap H^s_0(\T))$ for any $B \in \mathcal{B}(L^q_0(\T))$ needs to coincide with $\mu$ follows again from the identity $\mathcal{B}(H^s_0(\T)) = \mathcal{T}$. 
\end{proof}

To prove Theorem \ref{thm:im}, we will need a standard property of scalar conservation laws, namely the $L^1_0(\T)$-contraction. In the stochastic setting, we mention that a similar proof of the following proposition is done in \cite[Theorem 6.1]{Bor13}, but in the case where the flux function is $C^\infty$.

\begin{prop}[$L^1_0(\T)$-contraction]\label{L1C}
Under Assumptions \ref{ass:A} and \ref{ass:g}, let $(u(t))_{t\geq0}$ and $(v(t))_{t\geq0}$ be two strong solutions of \eqref{SSCL} starting from different initial conditions $u_0$ and $v_0$. Then, almost surely and for every $0\leq s\leq t$, we have
\[ \| u(t) - v(t) \|_{L^1_0(\T)} \leq \| u(s) - v(s) \|_{L^1_0(\T)} . \]
\end{prop}
\begin{proof}
We define a continuous approximation of the sign function by setting for all $\eta>0$,
\begin{equation*}
\sgn_\eta(u) :=
\begin{cases}
\frac u\eta , \quad &u\in[-\eta,\eta],\\
1,\quad &u\geq\eta,\\
-1,\quad &u\leq\eta,
\end{cases}
\end{equation*}
which gives rise to the following continuously differentiable approximation of the absolute value function:
\[ |v|_\eta:= \int_0^v \sgn_\eta(u)\dd u , \quad v \in \R . \]
Let $0\leq s \leq t$. We have
\begin{align}
&\int_\T \left| u(t)-v(t) \right|_\eta \dd x - \int_\T | u(s)-v(s)|_\eta \dd x = \int_\T \int_s^t \frac\dd{\dd r} \left| u(r)-v(r)\right|_\eta \dd r \dd x \label{regcontraction} \\
&= \int_\T \int_s^t \frac\dd{\dd r} \left( u(r)-v(r) \right) \sgn_\eta (u(r)-v(r)) \dd r \dd x \nonumber \\
&= \int_s^t \int_\T \left( A(u(r)) -A(v(r))-\nu\partial_x\left( u(r)-v(r)\right) \right) \partial_x \left( \sgn_\eta (u(r)-v(r)) \right) \dd x \dd r \nonumber \\
&\hspace{5cm}\text{(where we used the Fubini theorem and an integration by parts)} \nonumber \\
&= \int_s^t \int_\T \left( A(u(r)) -A(v(r))-\nu\partial_x\left( u(r)-v(r)\right) \right) \partial_x\left( u(r)-v(r)\right) \frac1\eta \ind{|u(r)-v(r)|\leq\eta}\dd x \dd r \nonumber \\
&\leq \int_s^t \int_\T (A(u(r))-A(v(r))) \partial_x (u(r)-v(r))\frac1\eta\ind{|u(r)-v(r)|\leq\eta}\dd x \dd r \nonumber
\end{align}
We fix
\[ M:= \sup_{r\in[s,t]} \|u(r)\|_{L^\infty_0(\T)} \vee \sup_{r\in[s,t]} \|v(r)\|_{L^\infty_0(\T)} ,\]
and we denote by $L_M$ a Lipschitz constant of $A$ over the interval $[-M,M]$. Since $(u(r))_{r\in[s,t]}$ and $(v(r))_{r\in[s,t]}$ belong to $C([s,t],H^2_0(\T))$ almost surely, then $M$ is finite almost surely and for all $r\in[s,t]$
\[ |A(u(r))-A(v(r))| |\partial_x(u(r)-v(r))| \frac1\eta \ind{|u(r)-v(r)|\leq\eta} \leq L_M \left|\partial_x(u(r)-v(r))\right| , \]
with
\[ \int_s^t \int_\T L_M \left| \partial_x(u(r)-v(r))\right| \dd x\dd r < +\infty . \]
Thus, we get from the dominated convergence theorem:
\begin{multline}
\lim_{\eta\to0} \int_s^t \int_\T (A(u(r))-A(v(r))) \partial_x (u(r)-v(r))\frac1\eta\ind{|u(r)-v(r)|\leq\eta}\dd x \dd r \\
= \int_s^t \int_\T \lim_{\eta\to0}(A(u(r))-A(v(r))) \partial_x (u(r)-v(r))\frac1\eta\ind{|u(r)-v(r)|\leq\eta}\dd x \dd r = 0 .
\end{multline}
As for the left-hand side of \eqref{regcontraction}, noticing that $|\cdot|_\eta$ increases to $|\cdot|$ as $\eta$ decreases, we have from the monotone convergence theorem
\[ \lim_{\eta\to0} \int_\T\left| u(t)-v(t) \right|_\eta \dd x = \left\| u(t)-v(t) \right\|_{L^1_0(\T)} , \qquad \lim_{\eta\to0}\int_\T | u(s)-v(s)|_\eta\dd x = \|u(s)-v(s)\|_{L^1_0(\T)} . \]
Hence, \eqref{regcontraction} yields the wanted result.
\end{proof}

\subsection{Existence}\label{existenceIM}

From the semigroup $(P_t)_{t\geq0}$ introduced in Subsection \ref{ss:mainresults}, we define its time-averaged semigroup $(R_T)_{T\geq0}$ by $R_0=\mathrm{Id}$, and for all $T>0$,
\[ R_T\varphi (u_0) = \frac1T \int_0^T P_t\varphi (u_0) \dd t , \qquad \varphi \in C_b(H^2_0(\T)) , \quad u_0\in H^2_0(\T) , \]
\[ \quad R^*_T\alpha (\Gamma) =\frac1T \int_0^T P^*_t\alpha(\Gamma)\dd t, \qquad \alpha \in \mathcal P (H^2_0(\T)) , \quad \Gamma \in \mathcal B (H^2_0(\T)). \]

Following the first part of Lemma \ref{lem:sigmafields}, for any $\alpha \in \mathcal P (H^2_0(\T))$ and $T \geq 0$, we denote by $\tilde{R}^*_T\alpha$ the Borel probability measure on $L^1_0(\T)$ defined by $\tilde{R}^*_T\alpha(\cdot) = R^*_T\alpha(\cdot \cap H^2_0(\T))$.

\begin{lemma}\label{tightness}
Under Assumptions \ref{ass:A} and \ref{ass:g}, for any $u_0 \in H^2_0(\T)$, there exists an increasing sequence $T^n \overset{n\to\infty}\longrightarrow +\infty$ and a probability measure $\tilde{\mu} \in \mathcal P(L^1_0(\T))$, such that the sequence of measures $(\tilde{R}^*_{T^n} \delta_{u_0})_{n\geq1}$ converges weakly to $\tilde{\mu}$ in $\mathcal P(L^1_0(\T))$.
\end{lemma}
\begin{proof}
Let $u_0 \in H^2_0(\T)$. From the inequality \eqref{intermapprox} with $p=2$, we can pass to the limit $r\to+\infty$ (which we recall implies that $T_r\to+\infty$ almost surely), and we get for all $T\geq0$,
\[ \E \left[ \int_0^T \| u(t) \|_{H^1_0(\T)}^2 \dd t \right] \leq \frac1{2\nu} \|u_0\|_{L^2_0(\T)}^2 +\frac{D_0T}{2\nu} . \]
Applying now the Markov inequality when $T\geq1$, we have for all $\varepsilon>0$,
\begin{equation}\label{tight}
\frac 1 T \int_0^T \P \left( \| u(t) \|_{H^1_0(\T)}^2 > \frac 1 \varepsilon \right) \dd t \leq \frac \varepsilon {2\nu} \left( \| u_0 \|_{L^2_0(\T)}^2 + D_0 \right) .
\end{equation}
Setting
\[ K_\varepsilon := \left\{v\in H^1_0(\T):\|v\|_{H^1_0(\T)}^2\leq\frac1\varepsilon \right\} , \]
we know from the compact embedding $H^1_0(\T)\subset\subset L^1_0(\T)$ that the set $K_\varepsilon$ is compact in $L^1_0(\T)$. Thus, rewriting \eqref{tight} as
\[ \tilde{R}^*_T\delta_{u_0}\left( L^1_0(\T)\setminus K_\varepsilon \right) \leq \frac\varepsilon{2\nu} \left( \| u_0 \|_{L^2_0(\T)}^2 + D_0 \right) , \]
we deduce that the family of measures $\{ \tilde{R}_T^*\delta_{u_0} : T\geq 1 \}$ is tight in the space $\mathcal P ( L^1_0(\T) )$. The result is then a consequence of Prokhorov's theorem \cite[Theorem 5.1]{Bil99}.
\end{proof}

\begin{lemma}\label{lem:existence}
Under the assumptions of Lemma \ref{tightness}, for all $p\geq1$, if $v$ is a random variable in $L^1_0(\T)$ distributed according to $\tilde\mu$, then
\[ \E \left[ \|v\|_{L^p_0(\T)}^p \right] < +\infty \qquad \text{and} \qquad \E\left[ \|v\|_{H^2_0(\T)}^2 \right] <+\infty . \]
Besides, the probability measure $\mu \in \mathcal P(H^2_0(\T))$ associated with $\tilde\mu$ by the second part of Lemma \ref{lem:sigmafields} is invariant for the semigroup $(P_t)_{t\geq0}$.
\end{lemma}
\begin{proof}
We start to show that the measure $\tilde{\mu} \in \mathcal P (L^1_0(\T))$ gives full weight to $H^2_0(\T)$. Thanks to Lemma \ref{H1approx}, since $T_r \underset{r\to\infty}\longrightarrow+\infty$ almost surely, we have:
\begin{equation}\label{H2bound}
\forall T>0 , \quad \frac1T \int_0^T \E \left[ \| u(s) \|_{H^2_0(\T)}^2 \right] \dd s \leq \frac1{T\nu} \left( \| u_0\|_{H^1_0(\T)}^2 + C_7 \left( 1+\|u_0\|_{L^{2p_A+2}_0(\T)}^{2p_A+2} \right) \right) + \frac{C_8}{\nu} .
\end{equation}
Let $(v_n)_{n\geq1}$ be a sequence of $H^2_0(\T)$-valued random variables such that $v_n \sim R_{T_n}^*\delta_{u_0}$ and $v_n$ converges in distribution in $L^1_0(\T)$ towards a random variable $v \sim \tilde{\mu}$. From \eqref{H2bound} and the definition of $(R_T)_{T\geq0}$, we have
\[ \limsup_{n\to\infty} \E \left[ \| v_n \|_{H^2_0(\T)}^2 \right] = \limsup_{n\to\infty} \frac1{T_n} \int_0^{T_n} \E_{u_0} \left[ \| u(s) \|_{H^2_0(\T)}^2 \right] \dd s \leq \frac{C_8}{\nu} . \]
Now, since $\|\cdot\|_{H^2_0(\T)}^2$ is lower semi-continuous on $L^1_0(\T)$, we get from Portemanteau's theorem:
\[ \E \left[ \| v \|_{H^2_0(\T)}^2 \right] \leq \liminf_{n\to\infty} \E \left[ \| v_n \|_{H^2_0(\T)}^2 \right] \leq \frac{C_8}\nu . \]
In particular, $v\in H^2_0(\T)$ almost surely, and thus $\tilde{\mu}$ gives full weight to $H^2_0(\T)$.

We now show that for any $p\geq1$, $\E [ \|v\|_{L^p_0(\T)}^p]<+\infty$. Let $p\geq1$. From Lemma \ref{momentapprox}, we have for all $T>0$,
\[ \frac1T \int_0^T \E_{u_0} \left[ \| u(s) \|_{L^p_0(\T)}^p \right] \dd s \leq \frac{C_5^{(p)}}T \left( 1+\|u_0\|_{L^p_0(\T)}^p \right) + C_6^{(p)} . \]
Once again, we use Portemanteau's theorem and the lower semi-continuity, this time of $\|\cdot\|_{L^p_0(\T)}^p$, on $L^1_0(\T)$:
\[ \E \left[ \| v \|_{L^p_0(\T)}^p \right] \leq \liminf_{n\to\infty} \E \left[ \| v_n \|_{L^p_0(\T)}^p \right] = \liminf_{n\to\infty} \frac1{T_n} \int_0^{T_n} \E_{u_0} \left[ \|u(s) \|_{L^p_0(\T)}^p \right] \dd s \leq C_6^{(p)} , \]
and the wanted result follows.

To prove the invariance of the measure $\mu$ with respect to $(P_t)_{t\geq0}$, we wish to apply the Krylov-Bogoliubov theorem \cite[Theorem 3.1.1]{DZ96}. However, $(P_t)_{t\geq0}$ is a Feller semigroup on the space $H^2_0(\T)$ (Corollary \ref{markov}) whereas our tightness result (Lemma \ref{tightness}) holds in $\mathcal P(L^1_0(\T))$. To overcome this inconvenience, we use Lemma~\ref{lem:sigmafields} and we place ourselves at the level of the embedded probability measures in $\mathcal P(L^1_0(\T))$, where we can adapt, thanks to Proposition~\ref{L1C}, the proof of \cite[Theorem 3.1.1]{DZ96}.

Let $\mu \in \mathcal P(H^2_0(\T))$ be associated with $\tilde\mu$ by the second part of Lemma~\ref{lem:sigmafields}, and let $\varphi \in C_b(L^1_0(\T))$. In particular, the restriction $\varphi_{|H^2_0(\T)}$ is bounded and continuous on $H^2_0(\T)$ and we can write
\begin{equation}\label{ex1}
\int_{H^2_0(\T)} \varphi \dd P^*_t \mu = \int_{H^2_0(\T)} P_t\varphi\dd\mu .
\end{equation}
It follows from the $L^1_0(\T)$-contraction property that the map $P_t\varphi : H^2_0(\T)\to\R$ is continuous with respect to the $L^1_0(\T)$-norm. To prove this fact, let $v_0 \in H^2_0(\T)$ and let $(v^{(j)}_0)_{j\geq1}$ be a sequence of $H^2_0(\T)$ such that $\|v^{(j)}_0-v_0\|_{L^1_0(\T)}\to0$, $j\to+\infty$. Let $(v(t))_{t\geq0}$ and $(v^{(j)}(t))_{t\geq0}$, $j\geq1$, be the strong solutions of~\eqref{SSCL} respectively with initial conditions $v_0$ and $v^{(j)}_0$, $j\geq1$. From Proposition~\ref{L1C}, we get almost surely and for all $t\geq0$,
\[ \lim_{j\to\infty} \left\| v^{(j)}(t) - v(t) \right\|_{L^1_0(\T)} = 0 . \]
Since $\varphi$ is bounded and continuous with respect to the $L^1_0(\T)$-norm, we have
\[ \lim_{j\to\infty} \left| P_t\varphi \left( v^{(j)}_0 \right) -P_t\varphi(v_0) \right| \leq \lim_{j\to\infty} \E \left[ \left| \varphi \left( v^{(j)}(t) \right) - \varphi(v(t)) \right| \right] = 0, \]
so that $P_t\varphi$ is continuous with respect to the $L^1_0(\T)$-norm.

As a consequence, from Lemma~\ref{tightness}, we have for all $t\geq0$
\begin{align*}
 \int_{H^2_0(\T)} P_t\varphi\dd\mu &= \int_{L^1_0(\T)} P_t\varphi \dd\tilde\mu \\
 &= \lim_{n\to\infty} \int_{L^1_0(\T)} P_t\varphi \dd\tilde R^*_{T^n}\delta_{u_0} \\
 &= \lim_{n\to\infty} \int_{H^2_0(\T)} P_t\varphi \dd R^*_{T^n}\delta_{u_0} \\
 &= \lim_{n\to\infty} \frac1{T^n} \int_0^{T^n} \int_{H^2_0(\T)} \varphi \dd P^*_{s+t} \delta_{u_0} \dd s \\
 &= \lim_{n\to\infty} \frac1{T^n} \int_t^{T^n+t} \int_{H^2_0(\T)} \varphi \dd P^*_s \delta_{u_0} \dd s \\
 &= \lim_{n\to\infty} \left( \frac1{T^n}\int_0^{T^n} \int_{H^2_0(\T)} \varphi \dd P^*_s \delta_{u_0} \dd s + \frac1{T^n} \int_{T^n}^{T^n+t} \int_{H^2_0(\T)} \varphi \dd P^*_s \delta_{u_0} \dd s - \frac1{T^n} \int_0^t \int_{H^2_0(\T)} \varphi \dd P^*_s \delta_{u_0} \dd s \right) \\
 &= \lim_{n\to\infty} \int_{H^2_0(\T)} \varphi \dd R^*_{T^n} \delta_{u_0} \\
 &= \lim_{n\to\infty} \int_{L^1_0(\T)} \varphi \dd\tilde R^*_{T^n}\delta_{u_0} = \int_{L^1_0(\T)} \varphi \dd\tilde\mu .
\end{align*}
For any $t\geq0$, $P^*_t\mu$ gives full weight to $H^2_0(\T)$ and therefore, following the first part of Lemma~\ref{lem:sigmafields}, we can define the associated Borel probability measure on $L^1_0(\T)$ by $\tilde P^*_t\mu=P^*_t\mu(\cdot\cap H^2_0(\T))$. From Equation~\eqref{ex1} and the above sequence of computations, it follows that for all $t\geq0$,
\[ \int_{L^1_0(\T)} \varphi \dd\tilde P^*_t\mu = \int_{L^1_0(\T)} \varphi \dd\tilde\mu , \]
Given that $\varphi$ has been chosen arbitrarily in $C_b(L^1_0(\T))$, this last equality says that $\tilde P^*_t\mu = \tilde\mu$. The second part of Lemma~\ref{lem:sigmafields} now ensures that $P^*_t\mu=\mu$.
\end{proof}

\subsection{Uniqueness}

The proof of the uniqueness part of Theorem \ref{thm:im} follows the ideas of the "small-noise" coupling argument from Dirr and Souganidis \cite{DS05}. On one hand, due to the dissipative nature of the drift, two solutions of~\eqref{SSCL} perturbed by the same noise and starting from different initial conditions are driven to balls of $L^2_0(\T)$ with small radius whenever this noise is small over sufficiently long time intervals. On the other hand, the $L^1_0(\T)$-contraction property ensures that when these two solutions get close to one another they stay close forever. Hence, each time the noise gets small enough, the two solutions get closer and closer and eventually, they show the same asymptotical behaviour. This idea allows to show that the law of two solutions have the same limit as the time goes to infinity. Therefore, starting from two invariant measures leads to the equality of these measures. The same kind of argument was used in \cite{DV15} for the invariant measure of kinetic solutions of inviscid scalar conservation laws and in \cite{Deb13} for the stochastic Navier-Stokes equations.

Let $(u(t))_{t\geq0}$ and $(v(t))_{t\geq0}$ be two solutions of \eqref{SSCL} driven by the same $Q$-Wiener process $(W^Q(t))_{t\geq0}$. For all $R>0$, we define the stopping time:
\[ \tau_R := \inf \left\{ t \geq 0 : \| u(t) \|_{H^1_0(\T)}^2 + \| v(t) \|_{H^1_0(\T)}^2 \leq R \right\} . \]

\begin{lemma}\label{entranceLB}
Under Assumptions \ref{ass:A} and \ref{ass:g}, there exists $R>0$ such that for any $u_0$ and $v_0$ in $H^2_0(\T)$, the stopping time $\tau_R$ is finite almost surely.
\end{lemma}
\begin{proof}
We can use here, from the statement of Lemma \ref{momentapprox}, the inequality \eqref{intermapprox} with $p=2$. In this case, we get
\[ 2\nu \E \left[ \int_0^{t\wedge\tau_R} \left( \|u(s)\|_{H^1_0(\T)}^2 +\|v(s)\|_{H^1_0(\T)}^2 \right)\dd s\right]\leq\|u_0\|_{L^2_0(\T)}^2+\|v_0\|_{L^2_0(\T)}^2+2D_0\E[t\wedge\tau_R], \]
from which we deduce, by definition of the stopping time $\tau_R$, that
\[ 2\nu R\E[t\wedge\tau_R]\leq\|u_0\|_{L^2_0(\T)}^2+\|v_0\|_{L^2_0(\T)}^2+2D_0\E[t\wedge\tau_R] . \]
Taking $R>D_0/\nu$ yields
\[ \E [ \tau_R ] = \lim_{t\to\infty} \E [ \tau_R \wedge t ] \leq \frac{\| u_0 \|_{L^2_0(\T)}^2 + \| v_0 \|_{L^2_0(\T)}^2}{2(\nu R-D_0)} < +\infty , \]
from which we derive the wanted result.
\end{proof}

The following result asserts that when the coupled processes $\left(u(t)\right)_{t\geq0}$ and $\left(v(t)\right)_{t\geq0}$ start from deterministic initial conditions inside some ball of $L^2_0(\T)$, then they both attain in finite time any neighbourhood of $0$ with positive probability:
\begin{lemma}\label{entrance}
Under Assumptions \ref{ass:A} and \ref{ass:g}, for any $M>0$ and any $\varepsilon>0$, there exist a time $t_{\varepsilon,M}>0$ and a value $p_{\varepsilon,M}\in(0,1)$ such that for all $u_0,v_0\in H^2_0(\T)$ satisfying $\|u_0\|_{H^1_0(\T)}^2+\|v_0\|_{H^1_0(\T)}^2\leq M$,
\[ \P \left( \left\| u(t_{\varepsilon,M}) \right\|_{L^2_0(\T)}^2 + \left\| v(t_{\varepsilon,M}) \right\|_{L^2_0(\T)}^2 \leq \varepsilon \right) \geq p_{\varepsilon,M} . \]
\end{lemma}
\begin{proof}
 Let $u_0$, $v_0 \in H^2_0(\T)$ be such that $\|u_0\|_{H^1_0(\T)}+\|v_0\|_{H^1_0(\T)}\leq M$, and let us define
 \[ t_{\varepsilon,M} = -\frac 1{2\nu} \log \left( \frac\varepsilon{4M} \right) . \]
 To prove the lemma, we are going to compare the trajectories of $\left(u(t)\right)_{t\geq0}$ and $\left(v(t)\right)_{t\geq0}$ with the trajectories of their noiseless counterparts $\left(\bar u(t)\right)_{t\geq0}$ and $\left(\bar v(t)\right)_{t\geq0}$, defined by
 \begin{equation*}
 \begin{cases}
 \partial_t \bar u (t) = - \partial_x A \left( \bar u(t) \right) + \nu \partial_{xx} \bar u(t) \\
 \bar u(0) = u_0
 \end{cases}
 \qquad
 \begin{cases}
\partial_t \bar v (t) = - \partial_x A \left( \bar v(t) \right) + \nu \partial_{xx} \bar v(t) \\
 \bar v(0) = v_0 .
 \end{cases}
 \end{equation*}
 
 Recall that the viscosity yields energy dissipation:
 \[ \frac \dd {\dd t} \left( \| \bar u(t) \|_{L^2_0(\T)}^2 + \| \bar v(t) \|_{L^2_0(\T)}^2 \right) = - 2\nu \left( \| \bar u(t) \|_{H^1_0(\T)}^2 + \| \bar v(t) \|_{H^1_0(\T)}^2 \right) . \]
 Applying~\eqref{strongerpoincare} on the right-hand side, we get
 \[ \frac \dd {\dd t} \left( \| \bar u(t) \|_{L^2_0(\T)}^2 + \| \bar v(t) \|_{L^2_0(\T)}^2 \right) \leq - 2\nu \left( \| \bar u(t) \|_{L^2_0(\T)}^2 + \| \bar v(t) \|_{L^2_0(\T)}^2 \right) , \]
 and we can now apply Gr\"onwall's lemma:
 \[ \| \bar u(t) \|_{L^2_0(\T)}^2 + \| \bar v(t) \|_{L^2_0(\T)}^2 \leq \left( \| u_0 \|_{L^2_0(\T)}^2 + \| v_0 \|_{L^2_0(\T)}^2 \right) \e^{-2\nu t} \leq M \e^{-2\nu t} . \]
 With our choice of $t_{\varepsilon,M}$, the above inequality means that as soon as $t\geq t_{\varepsilon,M}$, we have $\| \bar u(t) \|_{L^2_0(\T)}^2 + \| \bar v(t) \|_{L^2_0(\T)}^2 \leq \varepsilon/4$.
 
 Furthermore, it is a consequence of Lemma \ref{H1approx} that $(\bar u(t))_{t\geq0}$ satisfies
 \[ \left\| \bar u(t) \right\|_{H^1_0(\T)}^2 \leq \left\| u_0 \right\|_{H^1_0(\T)}^2 + C_7 \left(1+\left\| u_0\right\|_{L^{2p_A+2}_0(\T)}^{2p_A+2}\right) , \quad t\geq0. \]
 Indeed, when all the noise coefficients $g_k$ are equal to zero, the constant $C_8$ in the statement of Lemma \ref{H1approx} can also be taken equal to zero. Since the same inequality also applies to $(\bar v(t))_{t\geq0}$, we have
 \[ \left\| \bar u(t) \right\|_{H^1_0(\T)}^2 + \left\| \bar v(t) \right\|_{H^1_0(\T)}^2 \leq M+2C_7\left(1+M^{p_A+1}\right) =: C_9^{(M)} . \]

 We focus now on the trajectories of the random processes $\left(u(t)\right)_{t\geq0}$ and $\left(v(t)\right)_{t\geq0}$. We introduce the stopping time
 \[ \tilde\tau_M := \inf \left\{ t\geq0 : \| u(t) \|_{H^1_0(\T)} \vee \| v(t) \|_{H^1_0(\T)} \geq \frac12 + \sqrt{C_9^{(M)}} \right\} . \]
 Following Proposition~\ref{prop:mild-strong}, we may use the expressions of $(u(t))_{t\geq0}$ and $(\bar u(t))_{t\geq0}$ in the mild sense. From these mild formulations, we write
 \begin{equation}\label{mildineq}
 \left\| u(t) -\bar u(t) \right\|_{H^1_0(\T)} \leq \int_0^t \left\| S_{t-s} \partial_x \left( A(u(s)) - A(\bar u(s)) \right) \right\|_{H^1_0(\T)} \dd s + \| w(t) \|_{H^1_0(\T)} ,
 \end{equation}
where $(w(t))_{t\geq0}$ is the stochastic convolution associated with the $Q$-Wiener process $(W^Q(t))_{t\geq0}$. According to Proposition \ref{localboundedness}, we call $L_M$ a local Lipschitz constant of the map $z\in H^1_0(\T) \mapsto \partial_x A(z) \in L^2_0(\T)$ over the ball $\{ z \in H^1_0(\T) : \|z\|_{H^1_0(\T)}^2\leq\frac12+\sqrt{C_9^{(M)}} \}$ , and we place ourselves in the event
 \[ \left\{ \sup_{t\in[0,t_{\varepsilon,M}]} \| w(t) \|_{H^1_0(\T)} \leq \delta_{\varepsilon,M} \right\} , \qquad \text{where} \quad \delta_{\varepsilon,M} := \frac{\sqrt\varepsilon}{2\sqrt2} \frac1{1+2\sqrt{t_{\varepsilon,M}}C_4L_M} \e^{-C_4^2L_M^2t_{\varepsilon,M}} , \]
 where $C_4$ has been defined at Proposition \ref{localboundedness}. Taking $t\leq \tilde\tau_M \wedge t_{\varepsilon,M}$, applying the second part of Proposition~\ref{prop:heatkernel} and Proposition~\ref{localboundedness} to \eqref{mildineq}, we get
 \begin{align*}
 \left\| u(t) - \bar u(t) \right\|_{H^1_0(\T)} &\leq \int_0^t \frac{C_4}{\sqrt{t-s}} \left\| \partial_x(A(u(s)) - A(\bar u(s))) \right\|_{L^2_0(\T)} \dd s + \delta_{\varepsilon,M} \\
 &\leq \int_0^t \frac{C_4L_M}{\sqrt{t-s}} \| u(s)-\bar u(s) \|_{H^1_0(\T)} \dd s + \delta_{\varepsilon,M} .
 \end{align*}
Iterating this inequality and using the same arguments as in the proof of Lemma \ref{cd}, we get for all $t\leq t_{\varepsilon,M}\wedge\tilde\tau_M$,
 \[ \| u(t)-\bar u(t) \|_{H^1_0(\T)} \leq \delta_{\varepsilon,M} \left( 1+2\sqrt{t_{\varepsilon,M}\wedge\tilde\tau_M}C_4L_M\right) + C_4^2L_M^2\pi\int_0^t \| u(s)-\bar u(s) \|_{H^1_0(\T)} \dd s . \]
 Using now Gr\"onwall's lemma, we deduce
 \[ \| u(t) - \bar u(t) \|_{H^1_0(\T)} \leq \delta_{\varepsilon,M} \left( 1+2\sqrt{t_{\varepsilon,M}\wedge\tilde\tau_M}C_4L_M\right) \e^{C_4^2L_M^2\pi t} \leq \frac{\sqrt \varepsilon}{2\sqrt2} . \]
 Since the same arguments apply for the processes $\left(v(t)\right)_{t\geq0}$ and $\left(\bar v(t)\right)_{t\geq0}$, and given Equation~\eqref{strongerpoincare}, we have shown that for all $t\leq \tilde\tau_M \wedge t_{\varepsilon,M}$,
 \begin{align*}
 \| u(t) \|_{L^2_0(\T)}^2 + \| v(t) \|_{L^2_0(\T)}^2 &\leq 2 \left( \| \bar u(t) \|_{L^2_0(\T)}^2 + \| \bar v(t) \|_{L^2_0(\T)}^2 \right) + 2 \left( \| u(t)-\bar u(t) \|_{L^2_0(\T)}^2 + \| v(t)-\bar v(t) \|_{L^2_0(\T)}^2 \right) \\
 &\leq \frac\varepsilon 2 + \frac\varepsilon 2 = \varepsilon .
 \end{align*}

 We shall prove now that the event $\tilde\tau_M < t_{\varepsilon,M}$ is impossible. Indeed, assume for instance that $\left\|u\left(\tilde\tau_M\right)\right\|_{H^1_0(\T)} \geq \frac12 +\sqrt{C_9^{(M)}}$, then we would have
 \[ \left\| u\left(\tilde\tau_M\right)- \bar u\left(\tilde\tau_M\right) \right\|_{H^1_0(\T)} \leq \frac{\sqrt\varepsilon}{2} \quad \text{and} \quad \left\|\bar u\left(\tilde\tau_M\right)\right\|_{H^1_0(\T)}^2 \leq C_9^{(M)} , \]
 and thus,
 \[ \frac{\sqrt\varepsilon}{2} \geq \left\| u\left(\tilde\tau_M\right)- \bar u\left(\tilde\tau_M\right) \right\|_{H^1_0(\T)} \geq \left| \left\| u\left(\tilde\tau_M\right) \right\|_{H^1_0(\T)} - \left\| \bar u\left(\tilde\tau_M\right) \right\|_{H^1_0(\T)} \right| \geq  \left( \frac12 +\sqrt{C_9^{(M)}} \right) - \sqrt{C_9^{(M)}} =\frac12 , \]
 which is false for too small values of $\varepsilon$.
 
 We just have proven that for $M>0$ arbitrarily chosen and for all $u_0, v_0 \in H^2_0(\T)$ such that $\| u_0 \|_{H^1_0(\T)}^2 + \| v_0 \|_{H^1_0(\T)}^2 \leq M$, we have
 \[ \P \left( \| u(t_{\varepsilon,M}) \|_{L^2_0(\T)}^2 + \| v(t_{\varepsilon,M}) \|_{L^2_0(\T)}^2 \leq \varepsilon \right) \geq \P \left( \sup_{t\in[0,t_{\varepsilon,M}]} \| w(t)\|_{H^1_0(\T)} \leq \delta_{\varepsilon,M} \right) . \]
 To conclude the proof, it remains to check that
 \begin{equation}\label{SChitting}
 p_{\varepsilon,M} :=\P \left( \sup_{t\in[0,t_{\varepsilon,M}]} \| w(t) \|_{H^1_0(\T)} \leq \delta_{\varepsilon,M} \right) >0 .
 \end{equation}
We can write $\{ \sup_{t\in[0,t_{\varepsilon,M}]} \| w(t) \|_{H^1_0(\T)} \leq \delta_{\varepsilon,M} \} = \{ (w(t))_{t\in[0,t_{\varepsilon,M}]} \in B \}$ where $B$ is the closed ball of $C([0,t_{\varepsilon,M}],H^1_0(\T))$ with radius $\delta_{\varepsilon,M}$. Since the process $(w(t))_{t\in[0,t_{\varepsilon,M}]}$ is the mild solution to the stochastic heat equation (\textit{i.e.} Equation \eqref{strongsolution} with initial condition $w(0)\equiv0$ and flux $A\equiv0$), we can apply the support theorem from \cite[Theorem 1.1]{Nak04} which implies $\P( (w(t))_{t\in[0,t_{\varepsilon,M}]} \in B )>0$, so that \eqref{SChitting} is satisfied.
\end{proof}

\begin{lemma}\label{lem:uniqueness}
Under Assumptions \ref{ass:A} and \ref{ass:g}, any invariant measure $\mu$ for the process $(u(t))_{t\geq0}$ solution to \eqref{SSCL} is unique.
\end{lemma}
\begin{proof}
\textbf{Step 1. Almost sure confluence.} We start by fixing $\varepsilon>0$ small to which we associate the value $t_{\varepsilon,R}$ defined at Lemma \ref{entrance}, where $R$ has been defined at Lemma \ref{entranceLB}. We define the increasing stopping time sequence
 \[ \mathbf T_1 := \tau_R \]
 \[ \mathbf T_2 := \inf \left\{ t \geq \mathbf T_1 + t_{\varepsilon,R} : \| u(t) \|_{H^1_0(\T)}^2 + \| v(t) \|_{H^1_0(\T)}^2 \leq R \right\} \]
 \[ \mathbf T_3 := \inf \left\{ t \geq \mathbf T_2 + t_{\varepsilon,R} : \| u(t) \|_{H^1_0(\T)}^2 + \| v(t) \|_{H^1_0(\T)}^2 \leq R \right\} \]
 \[ \vdots \]
 
Lemma \ref{entranceLB} and the strong Markov property (Corollary \ref{markov}) ensure that every $\mathbf T_j$ is finite almost surely. We claim that
 \begin{equation}\label{induction}
 \forall J \in \N^* , \qquad \P \left( \forall j = 1 , \dots , J , \quad \left\| u(\mathbf T_j+t_{\varepsilon,R}) \right\|_{L^2_0(\T)}^2 + \left\| v(\mathbf T_j+t_{\varepsilon,R}) \right\|_{L^2_0(\T)}^2 > \varepsilon \right) \leq (1-p_{\varepsilon,R})^J .
 \end{equation}
 Indeed, it is true for $J=1$ thanks to the strong Markov property and Lemma \ref{entrance}:
\begin{align*}
&\P_{(u_0,v_0)} \left( \left\| u(\tau_R+t_{\varepsilon,R}) \right\|_{L^2_0(\T)}^2 + \left\| v(\tau_R+t_{\varepsilon,R}) \right\|_{L^2_0(\T)}^2 > \varepsilon \right) \\
&= \E_{(u_0,v_0)} \left[ \P_{(u_0,v_0)} \left( \left\| u(\tau_R+t_{\varepsilon,R}) \right\|_{L^2_0(\T)}^2 + \left\| v(\tau_R+t_{\varepsilon,R}) \right\|_{L^2_0(\T)}^2 > \varepsilon   | \mathcal F_{\tau_R} \right) \right] \\
&= \E_{(u_0,v_0)} \left[ \P_{\left(u(\tau_R),v(\tau_R)\right)} \left( \| u(t_{\varepsilon,R}) \|_{L^2_0(\T)}^2 + \|v(t_{\varepsilon,R})\|_{L^2_0(\T)}^2>\varepsilon \right) \right] \\
&\leq 1-p_{\varepsilon,R} ,
\end{align*}
and the general case follows by induction: assuming that inequality \eqref{induction} is true for some $J\in\N^*$, we have
\begin{align*}
 &\P_{(u_0,v_0)} \left( \forall j=1,\dots,J+1,\quad \| u(\mathbf T_j+t_{\varepsilon,R}) \|_{L^2_0(\T)}^2 + \| v(\mathbf T_j+t_{\varepsilon,R}) \|_{L^2_0(\T)}^2 > \varepsilon \right) \\
 &\quad = \E_{(u_0,v_0)} \left[ \P_{(u_0,v_0)} \left( \forall j=1,\dots,J+1,\quad \| u(\mathbf T_j+t_{\varepsilon,R}) \|_{L^2_0(\T)}^2 + \| v(\mathbf T_j+t_{\varepsilon,R}) \|_{L^2_0(\T)}^2 > \varepsilon | \mathcal F_{\mathbf T_{J+1}} \right) \right] \\
 &\quad = \E_{(u_0,v_0)} \left[ \left( \prod_{j=1}^J \ind{\| u(\mathbf T_j+t_{\varepsilon,R}) \|_{L^2_0(\T)}^2 + \| v(\mathbf T_j+t_{\varepsilon,R}) \|_{L^2_0(\T)}^2 > \varepsilon} \right) \P_{(u(\mathbf T_{J+1}),v(\mathbf T_{J+1}))} \left( \| u(t_{\varepsilon,R}) \|_{L^2_0(\T)}^2 + \| v(t_{\varepsilon,R}) \|_{L^2_0(\T)}^2 > \varepsilon \right) \right] \\
 &\quad \leq (1-p_{\varepsilon,R})^J \times (1-p_{\varepsilon,R}) = (1-p_{\varepsilon,R})^{J+1} .
\end{align*}
Taking the limit when $J$ goes to infinity, we get
\begin{align*}
 &\P \left( \forall j \in \N^* , \quad \| u(\mathbf T_j+t_{\varepsilon,R}) \|_{L^2_0(\T)}^2 + \| v(\mathbf T_j+t_{\varepsilon,R}) \|_{L^2_0(\T)}^2 > \varepsilon \right) \\
 &\quad = \lim_{J\to\infty} \P \left( \forall j=1,\dots,J, \quad \| u(\mathbf T_j+t_{\varepsilon,R}) \|_{L^2_0(\T)}^2 + \| v(\mathbf T_j+t_{\varepsilon,R}) \|_{L^2_0(\T)}^2 > \varepsilon \right) \\
 &\quad \leq \lim_{J\to\infty} (1-p_{\varepsilon,R})^J = 0 ,
\end{align*}
and consequently,
\begin{equation}\label{probasmallball}
\P \left( \exists t \geq 0 , \quad \| u(t) \|_{L^2_0(\T)}^2 + \| v(t) \|_{L^2_0(\T)}^2 \leq \varepsilon \right) = 1 .
\end{equation}
Since $\|u(t)-v(t)\|_{L^1_0(\T)}^2 \leq \|u(t)-v(t)\|_{L^2_0(\T)}^2 \leq 2 ( \|u(t)\|_{L^2_0(\T)}^2+\|v(t)\|_{L^2_0(\T)}^2)$ and since the value $\varepsilon>0$ has been chosen arbitrarily at the beginning of this proof, then Equality \eqref{probasmallball} means that almost surely,
\[ \forall \varepsilon>0, \quad \exists t \geq 0 , \quad \| u(t) - v(t) \|_{L^1_0(\T)}^2 \leq 2\varepsilon . \]
Recall however that Proposition \ref{L1C} states that almost surely, the mapping $t\mapsto \|u(t)-v(t)\|_{L^1_0(\T)}$ is non-decreasing. It follows that almost surely,
\begin{equation}\label{L1decay}
\lim_{t\to\infty} \| u(t)-v(t) \|_{L^1_0(\T)} = 0 .
\end{equation}

\textbf{Step 2. Uniqueness.} Let us now assume that there exist two invariant measures $\mu_1, \mu_2$ for the solution of \eqref{SSCL}, and let us take initial conditions $u_0$ and $v_0$ with distributions $\mu_1$ and $\mu_2$ respectively. For any test function $\phi : L^1_0(\T)\to\R$ bounded and Lipschitz continuous, we have for all $t\geq0$,
\[ \left| \E \left[ \phi(u_0) \right] - \E \left[ \phi(v_0) \right] \right| = \left| \E \left[ \phi(u(t)) \right] - \E \left[ \phi(v(t)) \right] \right| \leq \E \left[ \left| \phi(u(t))-\phi(v(t)) \right| \right] . \]
Since $\phi$ is Lipschitz continuous, from~\eqref{L1decay}, we have almost surely
\[ \lim_{t\to\infty} \left| \phi(u(t))-\phi(v(t)) \right| = 0 . \]
Moreover, for any $t\geq0$, we have almost surely $| \phi(u(t))-\phi(v(t)) | \leq 2 \sup |\phi|$. Thus, we may apply the dominated convergence theorem, which yields
\[ \left| \E \left[ \phi(u_0) \right] - \E \left[ \phi(v_0) \right] \right| \leq \lim_{t\to\infty} \E \left[ \left| \phi(u(t)) - \phi(v(t)) \right|\right] = 0 , \]
so that $\E [ \phi(u_0) ] = \E [ \phi(v_0) ]$, or in other words,
\begin{equation}\label{eq:phi}
  \int_{H^2_0(\T)} \phi \dd\mu_1 = \int_{H^2_0(\T)} \phi \dd\mu_2 .
\end{equation}
According to Lemma~\ref{lem:sigmafields}, let $\tilde\mu_1$ and $\tilde\mu_2$ be the probability measures on $\mathcal P(L^1_0(\T))$ associated to $\mu_1$ and $\mu_2$ respectively. Equation~\eqref{eq:phi} rewrites
\[ \int_{L^1_0(\T)} \phi \dd\tilde\mu_1 = \int_{L^1_0(\T)} \phi \dd\tilde\mu_2 , \qquad \forall \phi \in C_b\left( L^1_0(\T) \right), \]
so that $\tilde\mu_1=\tilde\mu_2$ and thus, by Lemma~\ref{lem:sigmafields}, $\mu_1=\mu_2$.
\end{proof}

\begin{proof}[Proof of Theorem \ref{thm:im}]
It follows from Lemmas \ref{lem:existence} and \ref{lem:uniqueness}.
\end{proof}

\subsection*{Acknowledgements}

The authors would like to thank S\'ebastien Boyaval for fruitful discussions and for his careful reading of this manuscript.

\bibliographystyle{plain}
\bibliography{biblio}

\end{document}